\documentclass{amsart}
\usepackage{amsmath, amsfonts, amsthm,amssymb,hyperref,graphicx,ulem,mathrsfs, ifpdf, color}
\input xy
\xyoption{all}

\newtheorem{theorem}{Theorem}[section]
\newtheorem{lemma}[theorem]{Lemma}
\newtheorem{proposition}[theorem]{Proposition}

\theoremstyle{definition}

\theoremstyle{remark}
\newtheorem{remark}[theorem]{Remark}

\numberwithin{equation}{section}

\begin{document}
\def\A{{\mathbb A}}
\def\C{{\mathbb C}}
\def\D{{\mathcal{D}}}
\def\E{{\mathscr{E}}}
\def\F{{\mathscr{F}}}
\def\G{{\mathscr{G}}}
\def\H{{\mathscr{H}}}
\def\I{{\mathcal{I}}}
\def\J{{\mathscr{I}}}
\def\K{{\mathcal{K}}}
\def\L{{\mathscr{L}}}
\def\M{{\mathscr{M}}}
\def\N{{\mathscr{N}}}
\def\O{{\mathcal{O}}}
\def\P{{\mathbb P}}
\def\Q{{\mathbb Q}}
\def\R{{\mathbb R}}
\def\S{{\mathscr{S}}}
\def\V{{\mathcal{V}}}
\def\Z{{\mathbb Z}}
\def\a{{\mathfrak{a}}}
\def\b{{\mathfrak{b}}}
\def\p{{\mathfrak{p}}}
\def\q{{\mathfrak{q}}}
\def\m{{\mathfrak{m}}}
\def\Cl{{\text{Cl }}}
\def\CL{{\text{Cl }}}
\def\Pic{{\text{Pic}}}
\def\Hom{{\text{Hom}}}
\def\Der{{\text{Der}}}
\def\deg{{\text{deg }}}
\def\Ext{{\text{Ext}}}
\def\Aut{{\text{Aut}}}
\def\Ab{{\mathfrak{Ab}}}
\def\Spec{{\text{Spec }}}
\def\Div{{\text{Div }}}
\def\coker{{\text{coker }}}

% \title[short text for running head]{full title}
\title{Towards the Classification of Weak Fano Threefolds with $\rho = 2$}

\author{Joseph W. Cutrone}
\address{Center for Talented Youth, Johns Hopkins University, 5801 Smith Ave, Baltimore, MD 21209}
\email{jcutrone@math.jhu.edu}
%\thanks{}

%    author two information
\author{Nicholas A. Marshburn}
\address{Department of Mathematics, Johns Hopkins University, 3400 N. Charles St., Baltimore, MD 21218}
%\curraddr{}
\email{marshbur@math.jhu.edu}
%\thanks{}

%    \subjclass is required.
\subjclass[2010]{}
%    The 2010 edition of the Mathematics Subject Classification is
%    now available.  If you are citing a classification from the
%    new scheme, use the following input coding instead.
%\subjclass[2010]{}

%    Abstract is required.
\begin{abstract}
In this paper, examples of type II Sarkisov links between smooth complex projective Fano threefolds with Picard number one are provided.  To show examples of these links, we study smooth weak Fano threefolds $X$ with Picard number two and with a divisorial extremal ray.  We assume that the pluri-anticanonical morphism of $X$ contracts only a finite number of curves.  The numerical classification of these particular smooth weak Fano threefolds is completed and the geometric existence of some numerical cases is proven.
\end{abstract}

\maketitle
\tableofcontents

\newpage
\section{Introduction}
The motivation for this paper\footnote{Partially supported by NSF DMS-1001427} starts from the last line in the second paragraph on page 2 of \cite{JPR07}: ``However, due to the complexity of the problem...we shall not consider the case that both contractions are birational, and hope to come back to that case later."  The authors seemed to have abandoned this project and this paper is our attempt to continue what they started.  In this paper, we finish the numerical classification of weak Fano threefolds $X$ with Picard number, denoted $\rho(X)$, equal to two and comment on the existence of these cases.  However, the geometric realization for several of these numerical cases still remains open.

Recall that a smooth threefold $X$ is \textit{Fano} if its anti-canonical divisor $-K_X$ is ample.  A variety $X$ is said to be a \textit{weak Fano} or an \textit{almost Fano} threefold if its anticanonical divisor $-K_X$ is both nef and big, but not necessarily ample. The classification of smooth Fano threefolds was completed by Fano, Iskovskih, Shokurov, Mori and Mukai.  For smooth Fano threefolds $Y$ with $\rho(Y) = 2$, there are 36 families (see \cite{IP99}). Although much progress has been made, the complete classification of smooth weak Fano threefolds $X$ with $\rho(X) = 2$ is still open.  In comparison, there are already over one hundred known possible families of smooth weak Fano threefolds, with about a hundred cases still open.  In \cite{JPR05}, complete classification was achieved when the anti-canonical morphism $\psi_{|-mK_X|}: X \rightarrow X'$ contracts a divisor.  In \cite{JPR07} and \cite{Tak09}, the case when the anticanonical morphism $\psi_{|-mK_X|}: X \rightarrow X'$ contracts only a finite number of curves was studied (i.e., when $\psi_{|-mK_X|}$ is \textit{small}).

The finite list of all possible smooth weak Fano threefolds with Picard number two that can numerically exist was determined for all combinations of Mori extremal contractions that can appear in a Sarkisov link except for the case when both contractions are divisorial. In \cite{JPR07}, the authors find the numerical classification and geometrically construct many cases when the Mori extremal contractions are not both birational.  It should be noted that their classification is incomplete.  In particular, the Bordiga threefold of degree 6 in $\P^5$ (see e.g., \cite{Ott92}) is known to have a structure of a $\P^1$-bundle but does not appear in the \cite{JPR07} classification. Takeuchi \cite{Tak09} completes the classification assuming one extremal contraction is a del Pezzo fibration of degree not equal to six.  Many cases with degree equal to six are still open and can be found in \cite{JPR07}. Our paper completes the numerical classification of smooth weak Fano threefolds $X$ with $\rho(X) = 2$ where both extremal contractions occurring in a Sarkisov link are divisorial. In addition, we comment on the existence of these cases, providing specific constructions for existence or proving non-existence.  Like the previous authors, the geometric realization of some of our new cases remains open.  Like the previous authors, we hope to return to these cases later.

The main theorem of this paper is the following:

\begin{theorem}
A smooth weak Fano threefold occurring as a central object of a Sarkisov link of type II between two terminal Fano threefolds with Picard number one must appear on our Tables with the corresponding numerical invariants.  These can all be found in Tables 1-9 in Section \ref{Tables}.
\end{theorem}

\begin{remark}
If both extremal morphisms contract a divisor to a curve, at present time there are 111 numerical links with 59 links proven to exist and 14 proven to not exist. These can be found in Tables 1-3  (\ref{tb:E1E1Table1}).  If both extremal morphisms contract a divisor to a point, there are 6 numerical links with 3 previously known to exist and one known not to exist. These can be found in Tables 7-9 (\ref{tb:E2E2Table}). If one extremal morphism contracts a divisor to a curve and the other extremal morphism contracts a divisor to a point, there are only 12 possible cases, all proven to geometrically exist.  These can be found in Tables 4-6 (\ref{tb:E1E2Table}).
\end{remark}

\subsection{Notations and Assumptions}
We will use similar notation as that of \cite{JPR07}.  Throughout this paper, we study complex projective threefolds $X$ satisfying the following assumptions:

\begin{tabbing}
\hspace*{.5cm} \=
    \+ i)   $X$ is smooth; \\
    ii)  $-K_X$ is nef and big (i.e., $X$ is a \textit{weak Fano} variety); \\
    iii) $X$ has finitely many $K_X$-trivial curves (i.e., $-K_X$ is big in codimension 1); \\
    iv)  $\rho(X)=2$; \\
    v) The linear system $|-K_X|$ is basepoint free;\\
    vi) The \textit{weak Fano index} $r_X$ of $X$ is 1 (i.e., $-K_X$ can not be written as $rH$ with $r > 1$).
\end{tabbing}

These varieties appear as smooth central objects of Sarkisov links between terminal Fano varieties with Picard number one. Conditions (i) - (iv) define smooth centers and conditions (v) and (vi) are special for our situation. Smooth centers not satisfying condition (v) and (vi) have already been classified. Classifying these links is a step toward classifying all birational maps between terminal Fano varieties with Picard number one.

Assumptions (ii) and (iv) above imply that $X$ has two extremal contractions: a $K_X$-negative contraction $\phi$ and a $K_X$-trivial contraction $\psi$. Assumptions (iii) and (v) imply that $\psi$ is a small nontrivial birational contraction induced by the linear system $|-K_X|$ (the \textit{anticanonical contraction}). By \cite{Ko89}, $\psi$ induces a flop $\chi$, and we obtain the following diagram:

\begin{equation}
\label{Firstfig}
\xymatrix{X \ar@{-->}[rr]^{\chi} \ar[d]^{\phi} \ar[dr]^{\psi}& & X^{+} \ar[d]^{\phi^{+}} \ar[dl]_{\psi^{+}} \\
          Y & X' & Y^{+}}
\end{equation}

In the above diagram, $\chi$ is a flop which is an isomorphism outside of the exceptional locus Exc($\psi$) and $X^{+}$ satisfies conditions (i)-(vi) above. The morphism $\phi^+$ is a $K_{X^+}$-negative extremal contraction and $\psi^{+}$ is the anticanonical morphism. The variety $X'$ is a terminal Gorenstein Fano threefold with Picard number one, but is not $\Q$-factorial since $\psi$ is small. Note that, in contrast to Fano varieties with Picard number two, we must perform a flop before obtaining a second extremal contraction.  The diagram represents a Sarkisov link of type II between the Mori fibrations $Y$/Spec $\C$ and $Y^+$/Spec $\C$.

\subsection{History}
Jahnke, Peternell, and Radloff \cite{JPR05} considered the case when the anticanonical contraction $\psi$ is divisorial.  These cases do not result in Sarkisov links and thus we make assumption (iii) above.

We are mainly interested in the cases where the extremal contractions on both sides of the diagram are divisorial.  The classification of these extremal rays was completed by Mori.
\begin{theorem} (Mori (1982))
\label{FirstMainThm}
Let $X$ be a smooth three dimensional projective variety.  Let $R$ be an extremal ray on $X$, and let $\phi: X \rightarrow Y$ be the corresponding extremal contraction.  Then only the following cases are possible:
\begin{enumerate}
\item $R$ is not numerically effective.  Then $\phi: X \rightarrow Y$ is a divisorial contraction of an irreducible exceptional divisor $E \subset X$ onto a curve or a point.  In addition, $\phi$ is the blow-up of the subvariety $\varphi(E)$ (with the reduced structure).  All the possible types of extremal rays $R$ which can occur in this situation are listed in the following table, where $\mu(R)$ is the length of the extremal ray $R$ (that is, the number min$\{(-K_X)\cdot C \, | \,C \in R$ is a rational curve$\})$ and $l_R$ is a rational curve such that $-K_X \cdot l_R = \mu(R)$ and $[l_R] = R$.
   $$
    \begin{tabular}{|c|l|c|l|}
      \hline
      Type of $R$ & $\phi$ and $E$  & $\mu(R)$ & $l_R$  \\ \hline \hline
      E1 & $\phi(E)$ is a smooth curve,                     & 1 & a fiber of a ruled surface $E$            \\
         & and $Y$ is a smooth variety                      &   &                                           \\ \hline
      E2 & $\phi(E)$ is a point, $Y$ is a                   & 2 & a line on $E \simeq \P^2$                 \\
         & smooth variety, $E \simeq \P^2$, and             &   &                                           \\
         & $\O_E(E) \simeq \O_{\P^2}(-1)$                   &   &                                           \\ \hline
      E3 & $\phi(E)$ is an ordinary double                  & 1 & $s \times \P^1$ or $\P^1 \times t$ in $E$ \\
         & point, $E \simeq \P^1 \times \P^1$ and           &   &                                           \\
         & $\O_E(E) \simeq \O_{\P^1 \times \P^1}(-1,-1)$    &   &                                           \\ \hline
      E4 & $\phi(E)$ is a double (cDV)-                     & 1 & a ruling of cone $E$                      \\
         & point, $E$ is a quadric cone                     &   &                                           \\
         & in $\P^3$, and $\O_E(E) \simeq \O_E \otimes$     &   &                                           \\
         & $\O_{\P^3}(-1)$                                  &   &                                           \\ \hline
      E5 & $\phi(E)$ is a quadruple non                     & 1 & a line on $E \simeq \P^2$                 \\
         & Gorenstein point on $Y$, $E$                     &   &                                           \\
         & $\simeq \P^2$, and $\O_E(E) \simeq \O_{\P^2}(-2)$ &   &                                          \\ \hline
    \end{tabular}
    $$
\item $R$ is numerically effective.  Then $\phi: X \rightarrow Y$ is a relative Fano model, $Y$ is nonsingular, dim $Y \le 2$, and all the possible situations are the following:
    \begin{enumerate}
    \item dim $Y = 2$: Then $\phi: X \rightarrow Y$ has a conic bundle structure, of which there are 2 types.
    \item dim $Y = 1$: Then $\phi: X \rightarrow Y$ has a del Pezzo structure, of which there are 3 types.
    \item dim $Y = 0$: Then $X$ is Fano.
    \end{enumerate}
\end{enumerate}
See \cite{IP99} for more details regarding the non divisorial cases.
\end{theorem}

\begin{lemma}
\label{FirstLem}Let $D$ be any divisor which is not $\psi$-nef, i.e., $-D$ is $\psi$-ample.  Then the $D$-flop of $\psi$ exists, i.e., a small birational map $\psi^+: X^+ \rightarrow X'$ such that the strict transform $\widetilde{D} \subset X^{+}$ is $\psi^{+}$-ample.  Moreover, $X^{+}$ is smooth with $-K_{X^{+}}$ big and nef and
$$
\begin{array}{c}
\rho(X^{+}) = 2; \\
(-K_X)^3 = (-K_{X^{+}})^3; \\
h^0(\O_X(D)) = h^0(\O_{X^{+}}(\widetilde{D})).
\end{array}
$$
\end{lemma}

\begin{proof} See Proposition 2.2 in \cite{JPR07}
\end{proof}

We can further assume that $-K_X$ is generated by global sections (assumption (v)).  The case when $-K_X$ is not generated by global sections is the case when $X'$ is a deformation of the Fano threefold $V_2$, a double cover of $\P^3$ ramified along a smooth quartic.  This is proved in Proposition 2.5 in \cite{JPR07}.

If $-K_X$ is divisible in Pic$(X)$, then $-K_X = r_XH$ for some $H \in \Pic(X)$, where $r_X$ is called the \textit{index} of $X$.  The divisor $H$ is called the \textit{fundamental divisor} and the linear system $|H|$ is the \textit{fundamental system} on $X$.  The self-intersection number $H^3$ is the \textit{degree} of $X$.  Since we are assuming $X$ to be smooth (and in particular Gorenstein), we remark that the index $r$ is a positive integer.  By using the smoothing of $X'$, \cite{Shi89} has shown that $r_X \le 4$, with equality when $X' = \P^3$.  In addition, \cite{Shi89} showed that when $r_X = 3$, $X' \subset \P^4$ is the quadric.  In both cases, $X \cong X^{+}$.  For $r_X = 3$, both $\phi$ and $\phi^{+}$ are $\P^2$-bundles over $\P^1$.  See \cite{JPR07} proposition 2.12 for more details.

The case $r_X = 2$ was treated in \cite{JP06}.  Both $\phi$ and $\phi^{+}$ are either $E2$-contractions, $\P^1$-bundles or quadric bundles. The complete list for the case when $\rho(X) = 2$ and $\psi$ small is given in \cite{JPR07} Theorem 2.13.  Thus we can assume that $r_X = 1$, which is assumption (vi) above.

Lastly, we note the result of Remark 4.1.10 in \cite{IP99} concerning the case when $X$ is hyperelliptic (that is, the anticanonical map $\varphi_{|-K_X|}$ is generically a double cover). If the anticanonical morphism is a double cover of a $\Q$-factorial threefold, then the flop $\chi$ is the birational automorphism of $X$ induced by $\varphi_{|-K_X|}$ (see e.g. \cite{JPR07} Lemma 2.8). In particular, $X$ is isomorphic to $X^+$ and $\phi$ and $\phi^+$ have the same type. This holds, for example, when $-K_X^3 = 2$, since $\varphi_{|-K_X|}$ is a double cover of $\P^3$.

\section{E1-E1 case}
\subsection{Equations and Bounds}
Let us consider first the case that both extremal contractions $\phi$ and $\phi^+$ are of type $E1$. Then $Y$ is a smooth Fano variety with Picard number 1, and $\phi$ is the blowup of a curve $C \subset Y$. Let $g$ and $d$ denote the genus and degree of $C$ respectively and let $E$ denote the exceptional divisor $\phi^{-1}(C)$. Denote by $H$ a fundamental divisor in $Y$. The pullback of $H$ to $X$ will also be denoted by $H$.  Since $\Pic(Y)$ is generated by $H$, $\Pic(X) \cong \Z \oplus \Z$ is generated by $H$ and $E$. We will instead use the divisors $-K_X$ and $E$, which do not generate Pic(X) unless the Fano index of Y is one. Unless the Fano index of $Y$ is one, these divisors do not generate $\Pic(X)$. However, they do generate $\Pic(X) \otimes \Q$, the Picard group of $X$ with coefficients in $\Q$.

The strict transform of a divisor $D \in \Pic(X)$ across the flop $\chi$ is denoted $\widetilde{D}$. Since $\chi$ is small, $\widetilde{K_X} = K_{X^+}$.  We identify divisors in $X$ and $X^+$ via $\chi$ and thus have an isomorphism between the Picard groups of $X$ and $X^+$. In particular, for any $D \in \Pic(X), \chi^{-1}(\chi(D)) = D$ and for any $D^+ \in \Pic(X^+), \chi(\chi^{-1}(D^+)) = D^+.$ In our notation, we can write this as $\widetilde{\widetilde{D}}=D$.

Note also that since $-K_X = \phi^*(-K_Y)-E = rH-E$, the denominators of the rational coefficients of the expression of a divisor in terms of $-K_X$ and $E$ will divide $r$.

Similarly we define $C^{+}$, $g^{+}$, $d^{+}$, $E^{+}, r^+$ and $H^+$.  We write

\begin{equation}
\label{Eplusequation}
\widetilde{E^{+}} = \alpha (-K_X) + \beta E
\end{equation}
for some nonzero rational numbers $\alpha$ and $\beta$ such that $r\alpha,r\beta \in \Z$.

Similarly, we let
\begin{equation}
\label{Eequation}
\widetilde{E} = \alpha^{+} (-K_{X^{+}}) + \beta^{+} E^{+}
\end{equation}
for some nonzero rational numbers $\alpha^{+}$ and $\beta^{+}$ such that $r^+\alpha^+,r^+\beta^+ \in \Z$.

By equating $E$ and $\widetilde{\widetilde{E}}$, as well as $E^+$ and $\widetilde{\widetilde{E^+}}$, and then comparing coefficients, we obtain the following relations:
\begin{equation}
\label{firstchecks}
\beta\beta^+ = 1,\,\,\, \alpha + \beta\alpha^+ = \alpha^+ + \beta^+\alpha = 0.
\end{equation}

The goal for the numerical classification of these cases is to find solutions to the Diophantine equations that result from comparing intersection numbers on both sides of the flop. Intersection with $-K_X$ is preserved, so we can create these equations by using the following well known formulas (see eg \cite{IP99}).  Here $r$ and $r^{+}$ are the Fano indices of $Y$ and $Y^{+}$ respectively. (Recall that the  weak Fano index of both $X$ and $X^{+}$ is assumed to be 1.)
\begin{equation}
\label{wellknowns}
\begin{array}{c}
E^3 = -rd + 2 - 2g; \\
K^2_X.E = rd + 2 -2g; \\
K_X.E^2 = 2 - 2g; \\
-K_Y^3 = -K_X^3 + 2rd+2-2g.
\end{array}
\end{equation}

Define
$$
\sigma := rd+2-2g = K_X^2.E.
$$

All these formulas and variables can be written inside of $X^{+}$ with the ``+" sign written accordingly. For example, $$\sigma^{+} := r^{+}d^{+}+2-2g^{+} = K_{X^+}^2.E^+.$$

Since the flop $\chi$ is a small transformation, it induces an isomorphism $\Pic(X) \cong \Pic(X^+)$. In the following lemma, we enumerate cases where this isomorphism preserves intersection numbers.

\begin{lemma} \label{intlemma}
For any divisors $D$ and $D'$ on $X$:
\begin{enumerate}
\item $K_X.D.D' = K_{X^+}.\widetilde{D}.\widetilde{D'}$
\item Let $C$ be a curve in $X$ such that $C$ is disjoint from every flopping curve in $X$. Then $D.C = \widetilde{D}.\widetilde{C}$ and $\widetilde{C}$, the strict transform of $C$ on $X^+$ under $\chi$, is disjoint from any flopping curves in $X^+$.
\end{enumerate}
\end{lemma}
\begin{proof}
Since $-K_X$ is assumed to be basepoint free, a general anticanonical divisor is disjoint from any of the finitely many flopping curves on $X$. In fact, the anticanonical divisors on $X$ are pullbacks of anticanonical divisors on $X'$, where the flopping curves are contracted to finitely many points.  The theorem then follows from the projection formula.
\end{proof}

Since intersection with $-K_X$ is preserved under the flop by the previous lemma, we have that
\begin{equation}
\label{intwithKx}
   \begin{array}{c}
            -K_{X^+}.(E^+)^2 = -K_X.\widetilde{E^+}^2 \\
            -K_X.(E^2) = -K_{X^+}.\widetilde{E}^2
      \end{array}
\end{equation}

These relations give us our Diophantine equations whose solutions give all possible numerical examples of the smooth weak Fano threefolds we are interested in. From (\ref{Eplusequation}) and (\ref{Eequation}), we can rewrite this as follows:
\begin{equation}
\label{C1forE1}
\begin{array}{cl}
    2g^+-2  & = -K_{X}(\alpha(-K_X)+\beta E)^2; \\
            & = \alpha^2(-K_X)^3+2\alpha\beta\sigma + \beta^2(2g-2).
\end{array}
\end{equation}
\begin{equation}
\label{C2forE1}
\begin{array}{cl}
    2g-2    & = -K_{X^+}(\alpha^+(-K_{X^+})+\beta^+ E^+)^2; \\
            & = (\alpha^+)^2(-K_X)^3+2\alpha^+\beta^+\sigma^+ + (\beta^+)^2(2g^+-2).
\end{array}
\end{equation}

In order to run a computer program to find solutions for these equations, we have to prove effective bounds for all variables involved. We continue to follow the ideas in \cite{JPR07}. Since $-K_X|_E$ is still nef and big, we have that $\sigma = K_X^2.E=(-K_X|_E)^2 > 0$. Similarly, $\sigma^{+} > 0$.
Since $X'$ has only terminal singularities, $X'$ is smoothable and the smoothing has the same Fano index as $X$, which is 1 by assumption (see e.g. \cite{JPR07} Prop. 2.4).  Since $|-K_X|$ is basepoint free, we have $$2 \le (-K_X)^3 \le 22 \,\,\,\textit{(evens only)}.$$
Since both $Y$ and $Y^{+}$ are smooth Fano threefolds of index $r$ and $r^{+}$ respectively, we have $1 \le r, r^{+} \le 4$.  By classification, we also have
\begin{equation} \label{KYbounds} 2 \le (-K_Y)^3 \le \left\{ \begin{array}{cc}22,& r = 1 \\40,& r = 2 \end{array}\right. \,\,\,\textit{(evens only) } \, \text{ and  } (-K_Y)^3 = \left\{ \begin{array}{cc} 54, & r = 3 \\ 64, & r = 4. \end{array} \right.
\end{equation}

The same bounds of course hold for $(-K_{Y^{+}})^3$. Next we bound both $d$ and $g$ (and apply the same argument to bound $d^{+}$ and $g^{+}$.)  From $22 \ge (-K_X)^3 = (-K_Y)^3 - \sigma - rd \ge 2$, we get
$$d \le \frac{(-K_Y)^3 - 3}{r} \le 19 \text{  and  } \sigma \le (-K_Y)^3 - rd - 2 \le \left\{ \begin{array}{cc}19, & r = 1 \\ 36, & r = 2 \\ 49, & r = 3 \\ 58, & r = 4.  \end{array} \right.$$

Finally, since $0< \sigma = rd - 2g + 2$ and by looking at each value of $r = 1, \ldots,4$ and the corresponding upper bound for $d$, we have $$g \le \frac{19r}{2} + 1.$$

Running a computer program alone at this point is impractical since we still have to loop through all rational values of $\alpha$ and $\beta$ (which are unbounded a priori at this point).  Thus the following lemma is necessary:
\begin{lemma} \label{betathm} Using the notation as above, if $X^+ \stackrel{\phi^+}{\longrightarrow} Y^+$ is an $E1$ contraction, then $\beta^{+} = \frac{-r}{r^{+}}$ and $\beta = \frac{-r^{+}}{r}$.
\end{lemma}
\begin{proof}
Since $\phi^+$ is of type $E1$, $-K_{X^{+}} = r^{+}H^{+} - E^{+}$.  Combining this with (\ref{Eequation}), we can rewrite $\widetilde{E}$ as follows: $$\begin{array}{cl} \widetilde{E} & = \alpha^{+}(r^{+}H^{+}-E^{+}) + \beta^{+}E^{+} \\ & = \alpha^{+}r^{+}H^{+} + (\beta^{+}-\alpha^{+})E^{+}. \end{array}$$

Thus $$\begin{array}{cl}  \Z / r\Z & \cong \Pic (X) / \langle -K_X, E \rangle \\
                                   & \cong \Pic (X^{+}) / \langle -K_{X^{+}}, \widetilde{E} \rangle \\
                                   & \cong \Pic (X^{+}) / \langle r^{+}H^{+} - E^{+}, \alpha^{+}r^{+}H^{+} + (\beta^{+} - \alpha^{+})E^{+} \rangle. \end{array}$$
Taking orders of both sides yields:
$$r = \left| \begin{array}{cc} \alpha^{+}r^{+} & r^{+}\\ \beta^{+}-\alpha^{+} & -1\end{array} \right|  = -\alpha^{+}r^{+} - r^{+}\beta^{+} + \alpha^{+}r^{+} = -r^+\beta^+$$ and thus $\beta^+ = \frac{-r}{r^{+}}$ as desired.  From (\ref{firstchecks}), it then follows that $\beta = \frac{-r^+}{r}$.
\end{proof}

Now that both $\alpha^+$, and then by (\ref{firstchecks}) $\alpha$, are completely determined by replacing (\ref{Eequation}) into the formula for $\sigma$:
$$ \sigma = K_X^2\widetilde{E} = \alpha^+(-K_{X^+})^3 + \beta^+K_{X^+}^2E^+.$$
Thus \begin{equation}
\label{alphapluseqn}\alpha^+ = \frac{\sigma - \beta^+\sigma^+}{(-K_X)^3}.
\end{equation}

We include the following two formulas for completeness:
\begin{equation}
\label{Ecubed}
\widetilde{E}^3 = (\alpha^+)^3(-K_X)^3 + 3(\alpha^+)^2\beta^+\sigma^+-3\alpha^+(\beta^+)^2K_X(E^+)^2 + (\beta^+)^3(E^+)^3.
\end{equation}
\begin{equation}
\label{Ecubed2}
\widetilde{E^+}^3 = \alpha^3(-K_X)^3 + 3 \alpha^2\beta\sigma - 3 \alpha\beta^2K_XE^2 + \beta^3E^3.
\end{equation}

All our variables are now bounded or completely determined via an explicit formula.  With our equations finalized, we are ready to run a computer program to find all the possible numerical values.  Programs were written in both Visual Basic and C++ to verify the results.  The Visual Basic source code can be found on the author's website. Here we summarize our bounds on the variables used in the program.

Variables and bounds:
$$\begin{array}{c}
1 \le r,r^+ \le 4; \\
2 \le K_X^3 \le 22 \,\,\text{(evens only)};\\
0 \le g \le \frac{19r}{2} + 1; \\
0 \le g^+ \le \frac{19r^+}{2} + 1; \\
1 \le d, d^+ \le 19.
\end{array}
$$

From these variables, the following are then determined:
$$\begin{array}{c}
-K_Y^3 = -K_X^3 + 2rd+2g-2; \\
-K_{Y^+}^3 = -K_X^3 + 2r^+d^++2g^+-2;\\
\sigma = rd+2-2g;\\
\sigma^+ = r^+d^++2-2g^+;\\
\beta^+ = -\frac{r}{r^+}; \\
\beta = \frac{1}{\beta^+};\\
\alpha^+ = \frac{\sigma - \beta^+\sigma^+}{-K_X^3};\\
\alpha = -\beta\alpha^+.\\
\end{array}
$$
Equations that any weak Fano threefold with our assumptions must satisfy (\ref{intwithKx}):
$$\begin{array}{cl} 2g^+-2      & = \alpha^2(-K_X)^3+2\alpha\beta\sigma + \beta^2(2g-2);\\
                    2g-2        & = (\alpha^+)^2(-K_X)^3+2\alpha^+\beta^+\sigma^+ + (\beta^+)^2(2g^+-2).
                    \end{array}
$$

\subsection{Numerical Checks}
Although the program can be run at this point, there are still numerical checks that can be added to the program to eliminate some possible solutions. Some simple numerical checks to include in the program have been previously mentioned, such as (\ref{firstchecks}) and $\sigma, \sigma^+ > 0$.  From the classification of smooth Fano threefolds $Y$ with $\rho(Y) = 1$, we also have to check that both $(-K_Y)^3, (-K_{Y^{+}})^3$ are even and that both these numbers satisfy the bounds from (\ref{KYbounds}).  In addition, since a smooth Fano threefold of Picard number 1 and degree 20 is known to not exist (see e.g., \cite{IP99}), we have to check that if $r = 1, (-K_Y)^3 \ne 20$ and similarly if $r^+ = 1, (-K_{Y^+})^3 \ne 20$.

Another obvious check to include in the program is that $r^3$ divides $(-K_Y)^3$ and that $(r^+)^3$ divides $(-K_{Y^+})^3$.  Using the formulas in (\ref{Ecubed}) and (\ref{Ecubed2}), we also need to make sure both $\widetilde{E}^3$ and $\widetilde{E^+}^3 \in \Z$.

Define the \textit{defect} of the flop to be:
\begin{equation}
\label{littlee}
\begin{array}{l}
e = E^3 - \widetilde{E}^3.
\end{array}
\end{equation}

Then we have the following two lemmas, which add two more checks to our growing list.
\begin{lemma}
\cite{Tak02} The correction term $e$ in (\ref{littlee}) is a strictly positive integer.
\end{lemma}
\begin{proof}
For a simpler proof then that of \cite{Tak02}, see the proof of Lemma 3.1 in \cite{Kal09}.
\end{proof}

The integer $e$ is not widely understood and is related to the number of flopping curves.  In fact, $e$ is equal to this number if the flop is a simple Atiyah flop and if for each flopping curve $\Gamma$ we have $H.\Gamma=1$.  At present, there is no known upper bound.  In our tables, we still have open examples with very large values of $e$.  If a strict upper bound could be found, then this bound could be used to eliminate the geometric realization of some open cases.

\begin{lemma}
The integer $r^3$ divides $e$.
\end{lemma}

\begin{proof}
Starting from the equality $E = rH + K_X$, take the strict transform of $E$ under $\chi$ to get $\widetilde{E} = (\widetilde{rH + K_X}) = r\widetilde{H} + K_{X^+}$.  Taking the difference of cubes in the formula for $e$ and using the facts that $-K_X^3 = -K_{X^+}^3$ and that intersection with $K_X$ is preserved under a flop (Lemma \ref{intlemma}) we get the desired result.
\end{proof}

\begin{lemma}If $\chi$ is an Atiyah flop, then $e = \displaystyle \sum_\Gamma (E.\Gamma)^3$, where the sum is taken over the finitely many flopping curves $\Gamma \subset X$.
\end{lemma}

\begin{proof}
We will prove the lemma in the case of a simple Atiyah flop. The general case is similar.  The flop $\chi$ is resolved by blowing up the flopping curve $\Gamma \subset X$ and contracting the resulting exceptional divisor $D \cong \P^1 \times \P^1 \subset Z$ with normal bundle $\O(-1) \oplus \O(-1)$ in a different direction.
$$
\xymatrix{ & Z \ar[dr]^{g} \ar[dl]_{f}& \\ X  \ar@{-->}[rr]^{\chi} & & X^+}
$$
Let $E_Z$ denote the strict transform of $E$ on $Z$.  Let $a = E.\Gamma$.  Then $f^*(E) = E_Z$ and $g^*(\widetilde{E})=E_Z+aD.$  We then have:
$$\begin{array}{cl}
e   & = E^3-\widetilde{E}^3 \\
    & = f^*(E)^3 - g^*(\widetilde{E})^3\\
    & = E^3_Z - (E_Z+aD)^3\\
    & = -3aE_Z^2.D-3a^2E_Z.D^2 - a^3D^3.
\end{array}$$
An easy check verifies  that $E_Z.D = 0, E_Z.D^2 = -a$, and $D^3 = 2$.  Thus $e = a^3$ as desired.
\end{proof}

Since $\widetilde{E^+} = \alpha(-K_X)+\beta E$, by replacing $-K_Y$ with $rH$ and $-K_X$ with $-K_Y - E = rH - E$, we have that \begin{equation}
\label{EplusSquiggle}
\widetilde{E^+} = \alpha rH + (\beta - \alpha)E.
\end{equation}

Since $\widetilde{E^+} \in \Pic (X)$ is not divisible, we have the following:
\begin{proposition} Using all notation as above, we have the following numerical checks:
\label{gcdchecks}
  $$  \begin{array}{l}
        \text{GCD}(\alpha r, \beta - \alpha) = 1; \\
        \text{GCD}(\alpha^+ r^+, \beta^+ - \alpha^+)=1; \\
        \alpha r, \alpha^+ r^+, \alpha - \beta, \alpha^+ - \beta^+ \in \Z.
      \end{array}
  $$

\end{proposition}

Since we are assuming $X$ is smooth, we can use the results of Batyrev and Kontsevich (see \cite{Ba}) which state that Hodge numbers are preserved under a flop.  The only interesting Hodge number in our situation is $h^{1,2}(X)$.  For smooth Fano threefolds $Y$ with $\rho(Y) = 1$, all these numbers are known (see \cite{IP99}).  In the $E1$ case, since we are blowing up a smooth curve of genus $g$, it can be shown (see \cite{CG72}) that $h^{1,2}(X) = h^{1,2}(Y) + g$.  Thus checking to see if $h^{1,2}(X) = h^{1,2}(X^+)$ is equivalent to checking the following equality, valid in the E1-E1 case:
\begin{equation}
\label{hodgecheck} h^{1,2}(Y) + g = h^{1,2}(Y^+) + g^+
\end{equation}

Using these checks to eliminate some possible cases from the output of the computer program yields Table (\ref{E1E1table}) in the Tables section of this paper.

\subsection{Elimination of Cases}
In this section, we will eliminate some of the numerical cases listed in Table (\ref{E1E1table}).

\begin{proposition}
\label{shtest1}
The following $E1-E1$ numerical cases in Table (\ref{E1E1table}) are not geometrically realizable: $$Nos. \, 27, 32,36,53,62,66,73,82,85,91,95.$$
\end{proposition}

\begin{proof}
We will show that case number $27$ can not exist.  A similar argument can be applied to the remaining cases.
From the data in case $27$, we have that
$$\begin{array}{cl}
        \widetilde{E^+} & \sim -4K_X - E \\
                        & \sim 4(-K_Y - E) - E \\
                        & \sim 4(2H -E) - E \\
                        & \sim 8H - 5E.
\end{array}$$
Since $\widetilde{E^+}$ is the strict transform of an exceptional divisor, $\widetilde{E^+}$ must be the unique member of the linear system $|\widetilde{E^+}|$.  However, since $$|\widetilde{E^+}| = |8H-5E| \supset 5|H-E| + 3|H|$$and since dim $|H| > 0$, the linear system $|H-E|$ must be empty.  We will show that this is in fact not the case.

To show that $|H-E| \ne \emptyset$, it is equivalent to show that $h^0(X,H-E) > 0$.  Since we are in the $E1$ case, this is in turn equivalent to showing that $h^0(Y,H-C) > 0$, where $C$ is the smooth curve of genus 0 and degree 1 being blown up.  $h^0(Y,H-C) > 0$ is equivalent to $h^0(Y,\I_C(H)) > 0$, which is what we will show.  Here $\I_C$ is the ideal sheaf of the curve $C$.

Start with the short exact sequence:
$$ 0 \rightarrow \I_C \rightarrow \O_Y \rightarrow \O_C \rightarrow 0.$$

Twist by $H$ to get the short exact sequence:
$$ 0 \rightarrow \I_C(H) \rightarrow \O_Y(H) \rightarrow \O_C(H) \rightarrow 0.$$

The corresponding long exact sequence of cohomology then gives us the exact sequence:
$$ 0 \rightarrow H^0(Y,\I_C(H)) \rightarrow H^0(Y,\O_Y(H)) \rightarrow H^0(Y,\O_C(H)) \rightarrow \ldots$$

To compute the dimensions of each of the vector spaces, we will use the Riemann-Roch formula.  From the well known formula for smooth (weak) Fano threefolds, we have:
$$h^0(jH) = \frac{j(r+j)(r+2j)}{12}H^3 + \frac{2j}{r} + 1$$ Plugging in our values of $j = 1$, $r = 2$ and $H^3 = -K^3_Y/r^3 = 8/2^3 = 1$, we get that $h^0(Y,\O_Y(H)) = 3$.

To compute $H^0(Y,\O_C(H))$ we will restrict to the curve $C$ and compute $H^0(C,\O_C(H))$ using the Riemann-Roch theorem for nonsingular curves. We then have $$ h^0(C,\O_C(H)) = \deg H|_C + 1 - g = 1 + 1 - 0 = 2.$$ Notice here that $h^1(C,\O_C(H))$ vanishes from Serre Duality, as the degree of $-H|_C + K_C$ is -3.  Since $h^0(Y,\O_Y(H)) = 3$ and $H^0(Y,\O_C(H)) = 2$, from the long exact sequence above we have $h^0(Y,\I_C(H)) \ge 1$ and thus the linear system $|H-E|$ is non-empty.
\end{proof}

\begin{proposition}
\label{e1e1case9}
Case No. 9 on the $E1-E1$ table (\ref{E1E1table}) does not exist.
\end{proposition}
\begin{proof}
By classification, $Y$ is embedded in $\P^6$ via |-$K_Y|$ and is the complete intersection of three quadrics.  The curve $C$ is of genus 1 and degree 3, but no such plane curve exists in an intersection of quadrics. Thus the case does not exist.
\end{proof}

\subsection{Geometric Realization of Cases}
In this section, we discuss the geometric existence of some cases found in Table (\ref{E1E1table}).

\begin{remark}
Some of the cases listed in Table (\ref{E1E1table}) were previously shown to exist as weak Fano varieties by Takeuchi.  In other instances, the methods used by Iskovskikh and other to show the existence of the smooth Fano Threefold with the given numerical invariants directly apply.  These cases are mentioned with their appropriate references in the corresponding tables.  In addition, some cases were recently shown to exist by Blanc and Lamy \cite{BL11} using different methods.
\end{remark}

\begin{proposition}
\label{MplusFmethod}
Case Nos. 29,33,49,50,51,52,75,76,89,90,98,99, and 103 on the $E1-E1$ table (\ref{E1E1table}) exist.
\end{proposition}

\begin{proof} We proceed case by case:

\textit{Case No. 29}:  Let $C$ be a curve of genus 0 and degree 5 inside of $Y_3 \subset \P^4$, a smooth cubic.  By \cite{Kn02}, $C$ is contained in a smooth K3-surface $S$, with $\Pic(S) = \Z H \oplus \Z C$.  By considering the short exact sequence of quadrics containing $C$, we see that the dimension of the linear system on $S$ of $|2H-C|$ is at least two.  Writing $|2H-C| = F + M$ as the sum of the fixed and movable components, we see that dim $|M| \ge 2$ and by \cite{Kn02}, deg $M \le 7$.  The linear system $|M|$ is not a pencil, since if it were, $2H-C = F + 2E$, where $F$ is a line and $E$ is an elliptic curve.  There would then be three contractions on $S$ of $C$, $F$, and $E$, contradicting the fact that $\rho(S) = 2$. By Bertini's theorem, we can choose a smooth curve $C'$ in $|M|$.  By classification and by \cite{Kn02}, our choices for the degree and genus of $C'$ are then $(g,d) \in \{(4,6),(2,6),(2,5)\}$.
After squaring both sides, the equation $2H - C = F + C'$ becomes $2 = F^2 + 2FC'+2g_{C'}-2$.  If $d_{C'} = 6$, we have then that $0 \ge FC' = 3 - g_{C'}.$ Positivity follows from the fact that $C'$ is nef.  Using this inequality, we can now eliminate case $(g,d) = (4,6)$.
For the remaining two cases, write $F = aH + bC$.  When the degree of $C' = 6$, $F$ is a line with $F^2 = -2$.  When the degree of $C'$ is 5, $F$ is a smooth irreducible conic and $F^2 = -2$.  Computing $F.C, F.C'$ and $F.H$ in these remaining cases allows us to solve for $a$ and $b$ and we see that both $a$ and $b$ are not integers as required.
Thus $M$ has degree 7 and $F$ is empty.  The linear system $|2H-C|$ on the $K3$-surface $S$ is then basepoint free by \cite{StD94} and blowing up $C$ gives the smooth weak Fano threefold that is case number 29.  By classification, the blowup $X$ of $_6C_9$ is not Fano and since $(X, \P^3,_6C_9)$ does not appear on any of the lists of \cite{JPR05} or \cite{JPR07}, case No. 29 exists.

\textit{Case No. 33}: Let $C$ be a curve of genus 2 and degree 6 in $Y_3 \subset \P^4$, a smooth cubic.  By \cite{Kn02}, we can pick a smooth K3-surface $S$ containing $C$ in $Y_3$ such that $\Pic(S) = \Z H \oplus \Z C$.  In particular, the degree of any curve on $S$ must be even.  Computing the dimension of the vector space of global sections, we see that dim $|2H-C| \ge 2$. Writing $|2H-C| = F + M$ as the sum of the fixed and movable components, we see that dim $|M| \ge 2$ and by \cite{Kn02}, deg $M \le 6$.  The linear system $M$ is not a pencil since the degree of the pencil would have to be at least 9.  Thus by Bertini's theorem, $|M|$ contains a smooth irreducible curve $C'$.  Again by \cite{Kn02} and the condition imposed on the degree of any curves in $S$ by our choice of generators of $\Pic S$, the degree of $C'$ must be 8 with genus 3.  Thus the fixed component $F$ of the linear system is empty and thus basepoint free by \cite{StD94}. Blowing up $C$ then gives a smooth weak Fano threefold $X$ with the numerical invariants of our case No. 33.  By classification, $X$ is not Fano and does not appear on any of the lists of \cite{JPR05} or \cite{JPR07}, case No. 33 exists.

\textit{Case No. 49}: Here $Y = \P^3$ and $C \subset \P^3$ is a curve of genus 2 and degree 8.  Using short exact sequences, we can find that $h^0(\P^3,\I_C(4)) \ge 4$ and thus $C$ is contained in a smooth quartic K3-surface $S$.  By \cite{Kn02}, we can write $\Pic(S) = \Z H \oplus \Z C$.  In particular, the degree of every curve on $S$ must be divisible by 4.
Since $h^0(\P^3,\I_C(4)) \ge 4$, dim $|4H-C|\big{|}_S \ge 2$.  Decompose $|4H-C| = F + M$, where $F$ is the fixed part and $M$ is the movable part.  Then dim $|M| \ge 2$ and by \cite{Kn02} deg $M \le 8$.  The linear system $|M|$ is not a pencil since $(4H - C)^2 > 0$.  If the degree of $M \ne 8$, by Bertini's theorem, we get another curve $C'$ of genus 3 and degree 4.  Any other possible combination of genus and degree were eliminated by our choice of generators for Pic S.  Since $C'$ is a hyperplane section, $C'$ is very ample.  Then writing $C' = H$, we have $4H-C = F + C'$, where $F$ is an irreducible rational curve of degree 4. Dotting this equation with $F$ gives $FC'=0$, a contradiction.
Thus $M$ has degree 8 and $F$ is empty.  The linear system $|4H-C|$ is then basepoint free by \cite{StD94} and blowing up the genus 2 degree 8 curve $C$ gives a smooth weak Fano threefold $X$, which is our case number 49. By classification, $X$ is not Fano and does not appear on any of the lists of \cite{JPR05} or \cite{JPR07}, and thus case No. 49 exists.

\textit{Case No. 50}: This case exists as the blow up of the residual curve of case No. 103.  See the section regarding Case No. 103 for details.

\textit{Case No. 51}: Consider a curve $C'$ of genus 2 and degree 6 inside a smooth k3-surface $Q$ of degree 4 in $\P^3$ by \cite{Kn02}.  By classification, $C'$ is not contained in a quadric, so $|2H-C| = \emptyset$.  Using short exact sequences, we see that $h^0(Q,I_{C'}(4)) \ge 3$ and thus $C$ is contained in a cubic surface.
Decompose $|3H-C'| = F + M$ into its fixed and movable components, with deg $M \le 6$ and dim $|M| \ge 2.$  If $|M|$ were a pencil, deg |M| = 6 and $F = \emptyset$.  Thus $|3H - C'|$ is basepoint free and therefore so is $|4H-C'|$.  Else by Bertini's theorem, we get another curve $C'' \in |M|$ of degree $d'' \le 6$ and genus $g_{C''} \ge 2$.
By \cite{Kn02}, $C''$ lies on a $K3$-surface and we only have to consider following two cases: $(g_{C''},d^{''}) = (2,5) $ or $(3,4)$.  in the first case (2,5), denote the fixed rational curve $F$ by $L$.  Since $C''$ is contained in a smooth quadric, $2H = C'' + L''$.  Therefore $3H-C' = L + 2H-L''.$  Dot with $L''$ to give 3 - C''.L'' = L.L'' + 2 + 2, and so $C'.L''= - 1 - L.L'' = -1 or -2$, which is a contradiction.  For the second case of (3,4), since $C''$ is contained in a plane, $M = H$ and so $3H-C' = F + H$.  Then $2H-C' = F$, contradicting the fact that $C'$ is contained in a quadric.
Therefore $d'' = 6$ and $F$ is empty. So $|3H-C'|$ is basepoint free and $C$ exists.  Since $|4H-C| \supset |3H-C|$, $|4H-C|$ is basepoint free.  By classification, $X$ is not Fano and does not appear on any of the lists of \cite{JPR05} or \cite{JPR07}, and thus case No. 51 exists.
It can be shown that the blow up of the original curve $C'$ with $g = 2$ and $d = 6$ gives rise to a weak Fano $X$ with $-K_X^3 = 18$, appearing on Table A7: Conic Bundle -Birational of \cite{JPR07}, Case No. 2.

\textit{Case No. 52}:  Let $C'$ be a curve of genus 2 and degree 5 inside a smooth K3 surface $Q_1$ of degree 4 in $\P^4$ by \cite{Kn02}.  Note that since $d \ge 2g+1$, $C'$ is projectively normal.  Using this fact and short exact sequences, we see that dim $|4H-C'| = 14$.  Since $C = |4H-C'|$, the degree of $C$ is 11.  Write $|4H-C'|$ as $M + F$, the sum of its movable and fixed components.  Then dim $|M| = 14$.  $|M|$ is not a pencil since if it were, its degree would be greater than 42, contradicting the fact that it has degree 11.  So by Bertini's theorem, we can then find a smooth irreducible curve which we will again denote by $M$.  If the fixed part $F$ exists, deg $M \le 10$. But by \cite{Kn02}, no curve of genus 14 and degree less than 10 can exits on a $K3$ surface.  Thus $C = M$ and $F = \emptyset$.  The linear system $|4H-C|$ is then basepoint free by \cite{StD94} and blowing up the genus 14 degree 11 curve $C$ gives a smooth weak Fano threefold $X$, which is our case number 52. By classification, $X$ is not Fano and does not appear on any of the lists of \cite{JPR05} or \cite{JPR07}, and thus case No. 52 exists.
It can be shown that by blowing up the other curve $C'$ of genus 2 and degree 5 gives the smooth Fano threefold No. 19 in \cite{IP99}.

\textit{Case No. 75}: Let $C \subset \P^3$ be a smooth curve of genus 3 and degree 8.  By \cite{Kn02}, $C$ lies on a smooth quartic K3-surface $S$ with $\Pic(S) =  \Z H \oplus \Z C$. In particular, the degree of any curve on $C$ must be divisible by 4.  Using short exact sequences, we can show that dim $|4H-C| \ge 3$.  Write $|4H-C| = M + F$ as the sum of its fixed and movable components.  Then dim $|M| \ge 3$ and deg $|M| \le 8$.  The linear system $|M|$ is not a pencil since the degree of the pencil would have to be at least 9.  Thus by Bertini's theorem, $|M|$ contains a smooth irreducible curve $C'$.  By \cite{Kn02} and the conditions imposed on the degree of curves in $S$ by our choice of generators of $\Pic(S)$, $C'$ must have degree 8 with genus 3. Thus the fixed component $F$ is empty and $|4H-C|$ is basepoint free.  Blowing up $C$ the gives a smooth weak Fano threefold $X$, which is our case number 75. By classification, $X$ is not Fano and does not appear on any of the lists of \cite{JPR05} or \cite{JPR07}, and thus case No. 75 exists.

\textit{Case No. 76}: This case is shown to exist in \cite{JPR05}, page 40, in their example No. 19.

\textit{Case No. 89}:  Let $C \subset \P^3$ be a smooth curve of genus 4 and degree 8.  By \cite{Kn02}, $C$ lies on a smooth quartic K3-surface $S$ with $\Pic(S) =  \Z H \oplus \Z C$. In particular, the degree of any curve on $C$ must be divisible by 4.  Using short exact sequences, we can show that dim $|4H-C| \ge 4$.  Write $|4H-C| = M + F$ as the sum of its fixed and movable components.  Then dim $|M| \ge 4$ and deg $|M| \le 8$.  The linear system $|M|$ is not a pencil since the degree of the pencil would have to be at least 12.  Thus by Bertini's theorem, $|M|$ contains a smooth irreducible curve $C'$.  By \cite{Kn02} and the conditions imposed on the degree of curves in $S$ by our choice of generators of $\Pic(S)$, $C'$ must have degree 8 with genus 3. Thus the fixed component $F$ is empty and $|4H-C|$ is basepoint free.  Blowing up $C$ the gives a smooth weak Fano threefold $X$, which is our case number 89. By classification, $X$ is not Fano and does not appear on any of the lists of \cite{JPR05} or \cite{JPR07}, and thus case No. 89 exists.

\textit{Case No. 90}: Let $C \subset \P^3$ be a curve of genus 0 and degree 7.  By \cite{Kn02}, $C$ lies on a smooth quartic K3-surface $S$ with $\rho(S) =2 $.  Using short exact sequences, we can show that dim $|4H-C| \ge 4.$  Write $|4H-C| = M + F$ as the sum of its fixed and movable components.  Then dim $|M| \ge 4$ and deg $|M| \le 9$.  The linear system $|M|$ is not a pencil since the degree of the pencil would have to be at least 12.  Thus by Bertini's theorem, $|M|$ contains a smooth irreducible curve $C'$.  Since $\rho(S) = 2$ and $C$ is rational, $F$ must be irreducible.  Otherwise, each component of $F$ could be contracted separately, giving $\rho(S) > 2$.  Then $F^2 = -2$, and squaring both sides of $4H - C = F + C'$ gives that $F.C' = 5 - g_{C'}$.  Since $C'$ is nef, $g_{C'} \le 5$.  From the classification of space curves, the only possible choices for genus and degree of $C'$ are $(d,g) \in \{(4,6),(4,7),(4,8),(5,7),(5,8)\}$.
If the degree of $C' = 8$, the $F$ is a line, the dimension of the linear system of hyperplanes containing $F$ is one, and the degree of $H - F = 3$.  Thus the residual curve $E = H -F$ is a cubic and the linear system is a pencil contracting that cubic.  Since there are then three possible curves that can be contracted, $E, C, $ and $F$, this contradicts $\rho(S) = 2$.
If the degree of $C' = 7$, then $F$ is an irreducible rational conic contained in a hyperplane.  Then dim $|H-F|=0$ and the degree $H-F = 2$.  The residual conic $C' = H-F$ can not be irreducible, else $\rho(S) \ne 2$, so $H$ must equal $2F$.  Squaring both side then gives $H^2 = -8$, a contradiction.
Lastly, if deg $C' = 6$, then $F$ is an irreducible rational cubic curve.  By counting dimensions, $C'$ must be contained in at least one quadric surface.  This $|2H-C'|$ would be an irreducible rational conic, contradicting the fact that $\rho(S) = 2$.
Therefore the degree of $C'$ must be 9 and the fixed component $F$ of $|4H-C|$ is empty.  By \cite{StD94}, the system is basepoint free. Blowing up $C$ the gives a smooth weak Fano threefold $X$, which is our case number 90. By classification, $X$ is not Fano and does not appear on any of the lists of \cite{JPR05} or \cite{JPR07}, and thus case No. 90 exists.

\textit{Case No. 98}: Let $C \subset \P^3$ be a curve of genus 1 and degree 7.  By \cite{Kn02}, $C$ lies on a smooth quartic K3-surface $S$ with $\rho(S) =2 $.  Using short exact sequences, we can show that dim $|4H-C| \ge 4.$  Write $|4H-C| = M + F$ as the sum of its fixed and movable components.  Then dim $|M| \ge 5$ and deg $|M| \le 9$.  The linear system $|M|$ is not a pencil since the degree of the pencil would have to be at least 15.  Thus by Bertini's theorem, $|M|$ contains a smooth irreducible curve $C'$.  By classification and \cite{Kn02}, the only possible choices for genus and degree of $C'$ are $(d,g) \in \{(5,8),(6,8),(6,7)\}$.
If deg $C' = 8$, then deg $F = 1$, with $H-F = E$, an irreducible rational cubic curve.  These three contractible curves $C,F,E$ would contradict $\rho(S) =2$.  If deg $C' = 6$, the $F$ would be an irreducible conic and $H-F=F$, or equivalently $H = 2F$.  Squaring both sides gives $H^2 < 0$, a contradiction.  Thus the degree of $C'$ must be 9 and the fixed component $F$ must be empty.  By \cite{StD94}, the system is basepoint free.  Blowing up $C$ the gives a smooth weak Fano threefold $X$, which is our case number 98. By classification, $X$ is not Fano and does not appear on any of the lists of \cite{JPR05} or \cite{JPR07}, and thus case No. 98 exists.

\textit{Case No. 99}:  The arguments used as in the previous cases apply, but this case also appears in \cite{JPR05}.  This case was shown to exist in \cite{BL11}.

\textit{Case No. 103}: Let $C \subset \P^3$ be a curve of genus 2 and degree 7.  By \cite{Kn02}, $C$ lies on a smooth quartic K3-surface $S$ with $\rho(S) =2$ and $\Pic(S) =  \Z H \oplus \Z C$.  Using short exact sequences, we can show that dim $|4H-C| \ge 6.$  Write $|4H-C| = M + F$ as the sum of its fixed and movable components.  Then dim $|M| \ge 6$ and deg $|M| \le 9$.  The linear system $|M|$ is not a pencil since the degree of the pencil would have to be at least 18.  Thus by Bertini's theorem, $|M|$ contains a smooth irreducible curve $C'$.  By classification and \cite{Kn02}, the only possible choices for genus and degree of $C'$ are $(d,g) \in \{(6,8),(7,8),(8,8),(9,8),(6,7),(7,7)\}$.
If deg $C'$ = 8, after both sides of $4H-C = F + C'$, we get $0 \ge F.C' = 7 - g_{C'}$, where positivity following from $C'$ being nef.  This eliminates cases $(g,d) = (8,8)$ and $(9,8)$.
For the remaining cases, write $F = aH + bC$.  When deg $C' = 8$, $F$ is a line with $F^2 = 1$.  When deg $C' = 7$, $F$ is a smooth irreducible conic of the sum of two lines and thus $F^2 = -2$ or $-4$.  Computing $F.C, F.C'$ and $F.H$ in all these remaining four cases and using these numbers to solve for $a$ and $b$ show that $a,b \not\in \Z$, a contradiction.
Thus $M$ has degree 9 and $F$ is empty.  By \cite{StD94}, the system is basepoint free.  Blowing up $C$ the gives a smooth weak Fano threefold $X$, which is our case number 103. By classification, $X$ is not Fano and does not appear on any of the lists of \cite{JPR05} or \cite{JPR07}, and thus case No. 103 exists.
It can be checked that the residual curve $C'$ of $C$ is genus 6 and degree 9. Blowing up $C'$ the gives a smooth weak Fano threefold $X$, which is our case No. 50. By classification, $X$ is not Fano and does not appear on any of the lists of \cite{JPR05} or \cite{JPR07}, and thus case No. 50 exists.

\textit{Case No. 107}:
\label{e1e1107}This case was shown to exist by the results of I. Cheltsov and C. Shramov in their upcoming paper \cite{ChSh11}.
\end{proof}

We separate this example from the rest due to the different methods used than those in the prior examples.

\begin{proposition}
\label{ex111}
Case number $111$ on the $E1-E1$ table (\ref{E1E1table}) exists.
\end{proposition}
\begin{proof}
Let $S \subset \P^3$ be a nonsingular cubic surface, the blow-up of $\P^2$ at six general points (no three on a line, no six on a conic).  Let the Picard group of $S$ be generated by $l, e_1, \ldots, e_6$, where $l$ is the pullback of a line in $\P^2$ and $e_1, \ldots, e_6$ are the exceptional divisors.  The intersection numbers between the generators of $\Pic(S)$ are $l^2 = 1, l.e_i = 0$ for any $i = 1, \ldots,6$, and $e_i.e_j = \delta_{ij}$, where $\delta_{ij}$ is the Kronecker delta. In $S$, consider the divisor $$C \sim 7l - 4e_1 - 3e_2-3e_3-2e_4-2e_5-2e_6.$$  Since for any of the 27 lines $L$ on $S$, $C.L \ge 0$ and $C^2 >0$.  Thus by the theory of cubic surfaces, the linear system $|C|$ contains an irreducible nonsingular member which we will also denote by $C$.  The degree of any effective divisor $D \sim al - \sum b_ie_i$ on $S$ as a curve in $\P^3$ is $3a - \sum b_i$, and the genus of $D$ is $\frac{1}{2}(a-1)(a-2) - \frac{1}{2}\sum (b_i^2 - b_i)$.  Thus the degree and genus of $C$ are 5 and 0 respectively in accordance with the given data. Let $X$ be the blowup of $C$ in $\P^3$.  Let $\widetilde{S}$ in $X$ be the strict transform of the cubic surface $S$.  Then $\widetilde{S} \in |3H-E| = |-K_X - H|$. Suppose $\Gamma \subset X$ is a curve such that $-K_X.\Gamma \le 0$.  Then $\widetilde{S}.\Gamma < 0$ since $\widetilde{S} \sim -K_X - H$ and thus $\Gamma$ is contained in $\widetilde{S}$.  Now $\widetilde{S} \cong S$, so write $\Gamma \sim al - \sum b_i e_i$. We now claim that $\Gamma$ must be a line in the cubic surface $\widetilde{S}$.  To show this, it suffices to show that $\Gamma^2 = a^2 - \sum b_i^2 < 0$, since the lines on a cubic surface are the only divisors with negative self-intersection.
Starting with $-K_X.\Gamma \le 0$, we rewrite this inequality as $(4H-E).\Gamma \le 0$ or equivalently $4H.\Gamma \le E.\Gamma$.  Viewing the intersection on the right in $S$, we have $4 \, \text{deg}_{\P^3} \Gamma \le C.\Gamma$.  This is equivalent to
$$4(al-\sum b_i e_i).(3l - \sum e_i) \le (7l - 4e_1 - 3e_2-3e_3-2e_4-2e_5-2e_6).(al - \sum b_ie_i) \\$$
which is equivalent to $$5a \le b_2 + b_3 + 2b_4 + 2b_5 + 2b_6.$$
If $\Gamma$ is $e_1$, we have $-K_X.\Gamma = 0$. If $\Gamma = e_i$ for $i = 2, \ldots, 6$ we get that $-K_X.\Gamma >0$.  So we can assume $\Gamma$ is not one of the $e_i$'s.  Since $\Gamma$ is effective, it follows that $5a \le 2\sum b_i$.  Squaring both sides gives $25a^2 \le 4(\sum b_i)^2$.  Using the Cauchy-Schwartz inequality, we have $$25a^2 \le 4(\sum b_i)^2 \le 4 \cdot 6 \sum b_i^2 < 25\sum b_i^2.$$Canceling the 25 yields the desired $$a^2 < \sum b_i^2$$
Since we now know that $\Gamma$ is one of the possible 27 lines on a cubic surface, a simple check shows that the only curve which is $-K_X$-trivial is the exceptional divisor $e_1$, with all other lines being $-K_X$-positive.  Since $e/r^3 =1$, it is expected that there is only 1 flopping curve. Therefore $X$ is a weak Fano threefold. Since this example is the only case on all tables with genus 0 and degree 5, this must be case 111 on our table.
\end{proof}

\begin{remark}
\label{e1e1quadrics}
The cases when $X$ is the blow up of a smooth curve on a smooth quadric $Y \subset \P^4$, i.e., case Nos. 44, 45, 46, 48, 70, 86, 88, 97, and 105 all exist.  In addition, the cases when $X$ is the blow up of a curve on the smooth intersection of two quadrics, i.e., case Nos. 30, 34, 37, 41, 64, and 84, all exist. These and others are the subject of the paper \cite{ACM11}.
\end{remark}

\section{E1-E* Case (*= 2,3,4,5)}
In this section, we classify cases of type $E1 - E2, E1 - E3/E4$ and $E1-E5$.  The diagram (\ref{Firstfig}) and all notations in the the previous section still apply.  We assume that the left side of the diagram $\phi:X \rightarrow Y$ is the $E1$ contraction and the right side $\phi^{+}:X^+ \rightarrow Y^+$ is the $E2,E3/E4$ or $E5$ contraction.  Since the a contraction of type $E3$ is numerically equivalent to a contraction of type $E4$, we will not distinguish between them in what follows.

\subsection{Equations and Bounds}
As before, the goal is to use the relations in (\ref{intwithKx}) to find Diophantine equations, find all solutions using a computer program, and then check their geometric realization.  The first observation to note is that regardless of what the contraction $\phi^+: X^+ \rightarrow Y^+$ is, we always have the following relation:
\begin{equation}
\label{KXEplus2}
K_X.(\widetilde{E^+})^2 = 2.
\end{equation}

Using the known values in each case for $E^+|_{E^+}$ in Theorem \ref{FirstMainThm}, we can easily see the above equality using adjunction as follows:
$$
\begin{array}{cl}
    K_X.(\widetilde{E^+})^2  & = K_{X^+}.(E^+)^2 \\
                                & = K_{X^+}|_{E^+}.E^+|_{E^+} \\
                                & = (K_{E^+} - E^+|_{E^+}).E^+|_{E^+} \\
                                & = 2.
\end{array}
$$

For each non $E1$ contraction, it is easy to show the following intersection numbers using similar reasoning as above (see Lemma 4.1.6 in (\cite{IP99})):
$$
(-K_{X^+})^2.E^+ =  \left\{ \begin{array}{cl}   4 & \text{   if } E2 \\
                                                2 & \text{   if } E3/E4 \\
                                                1 & \text{   if } E5 \end{array} \right.
$$

Now from the following equalities:
\begin{equation}
\label{E1Eints}
\begin{array}{c}
K_X.(\widetilde{E^+})^2 = 2; \\
K_X^2.\widetilde{E^+} = K_{X^+}.E^+; \\
K_X.E^2 = K_{X^+}.(\widetilde{E})^2; \\
K_X^2.E = K_{X^+}^2.\widetilde{E}.
\end{array}
\end{equation}
using (\ref{Eplusequation}), (\ref{Eequation}), and (\ref{wellknowns}), we have the following system of Diophantine equations:
\begin{equation}
\label{E1Eequations}
\begin{array}{c}
\alpha^2 (K_X)^3-2\alpha \beta r d + (2 - 2g)(-2\alpha \beta + \beta^2) = 2.                         \\
\alpha(-K_X)^3 + \beta (rd + 2 - 2g) = \left\{ \begin{array}{cl}   4 & \text{   if } E2 \\
                                            2 & \text{   if } E3/E4 \\
                                            1 & \text{   if } E5 \end{array} \right.             \\
2-2g = (\alpha^+)^2(K_X)^3 - 2\alpha^+ \beta^+ \cdot \left\{ \begin{array}{cl}   4 & \text{   if } E2 \\
                                            2 & \text{   if } E3/E4 \\
                                            1 & \text{   if } E5 \end{array} \right.             + 2 (\beta^+)^2. \\
rd+2-2g = \alpha^+(-K_X)^3 + \beta^+ \cdot \left\{ \begin{array}{cl}   4 & \text{   if } E2 \\
                                                                 2 & \text{   if } E3/E4 \\
                                                                 1 & \text{   if } E5 \end{array} \right.             \\

\end{array}
\end{equation}

To run a computer program, we still have to compute the necessary bounds on $\alpha^+$ and $\beta^+$ since Lemma \ref{betathm} only applies if $\phi^+: X^+ \rightarrow Y^+$ is an $E1$ or $E2$ contraction.

We first bound $\beta^+$.  From (\ref{gcdchecks}) it follows that $r\beta \in \Z$. Using (\ref{firstchecks}), this becomes $\frac{r}{\beta^+} \in \Z$, so $|\beta^+| \le r$.

Using the last Diophantine equation above and the bounds for $d$ and $g$ computed in the previous section, we can now bound $\alpha^+$.  Since $\sigma > 0$, we have
$$rd + 2 - 2g = 2rd - \sigma +4 -4g \le 2rd +4 \le 2 \cdot 4 \cdot 19 + 4 = 156$$
Therefore $$\sigma^+(-K_{X})^3 + \beta^+ \cdot \left\{ \begin{array}{cl}4 & \text{   if } E2 \\
                                                                 2 & \text{   if } E3/E4 \\
                                                                 1 & \text{   if } E5 \end{array} \right. \le 156$$
Rearranging the terms, and using the fact that $|\beta^+|\le r$, we have that $$\alpha^+(-K_X)^3 \le 156 + 16 = 172$$ and so using the lower bound on $(-K_X)^3$, we have now shown the following:
\begin{proposition}
\label{E2alphabounds}
For an E2, E3/E4 or E5 contraction, $0< \alpha^+ \le 86$ and $|\beta^+| \le r$.  Since $\beta^+ < 0$, $-r \le \beta^+ \le -1$
\end{proposition}

We need one more proposition:
\begin{proposition}
\label{E2alphaIsInt}
If $\phi: X \rightarrow Y$ is an $E2, E3/E4$ or $E5$ type contraction, then $\alpha$ and $\beta$ as in (\ref{Eplusequation}) are integers.
\end{proposition}
\begin{proof}
\textit{Case 1:} The contraction $\phi$\textit{ is E2}: If the Fano index of $Y$, $r_{Y}$, is two, then the Fano index of $X$, would be two, contradicting our assumption that $r_X = 1$.  By classification of smooth Fano threefolds with Picard number one, if $r_{Y} = 3$ or 4, then $Y$ is either the smooth quadric $Q \subset \P^4$ or $\P^3$ respectively.  In both of these situations, the blow up of a point would be a Fano variety.  Thus we only need to consider the case when $Y$ has Fano index one. Let $l$ be a line on $Y$ which we know to exist by the classic theorem of Shokurov.

Then $-K_{X}.\phi^*(l) = -K_{Y}.l =1$ and $E.\phi^*(l) = 0$.  Therefore $\Z \ni \widetilde{E^+}.\phi^*(l) = (\alpha(-K_{X}) + \beta E).\phi^*(l) = \alpha$.  Therefore $\widetilde{E^+} + \alpha K_{X} = \beta E$ is Cartier, so $\beta \in \Z$.\\

\textit{Case 2:} $\phi$\textit{ is E3/E4 or E5}: Let $\psi$ be the flopping contraction.  Then $\psi$ restricted to $E$ is a finite birational morphism.  The linear system corresponding to the morphism $\psi|_{E}$ is the subsystem of $|-K_{X}|_{E}|$ corresponding to the image of $H^0(X,-K_{X}) \rightarrow H^0(X,-K_{X}|_{E})$.  By Mori's theorem classifying extremal rays (\ref{FirstMainThm}), we know that in the $E3/E4$ case, $\O_{X}(-K_{X})|_{E} \cong \O_{Q}(1)$ on a quadric $Q \subset \P^3$ and for the $E5$ case, $\O_{X}(-K_{X})|_{E} \cong \O_{\P^2}(1)$ on $\P^2$.  No proper subsystem of either of the complete linear systems associated to these sheaves gives a finite birational morphism, so $H^0(X,-K_{X}) \rightarrow H^0(X,-K_{X}|_{E})$ is surjective.  Therefore the linear system of $\psi|_{E}$ is $|-K_{X}|_{E}|.$ Thus $\psi|_{E}$ is an embedding since $-K_{X}|_{E} = \O(1)$ is very ample.  This implies the intersection of $E$ with any flopping curve $\Gamma$ must be transversal at a single point, ie $E.\Gamma =1$.  Then $\Z \ni \widetilde{E^+}.\Gamma = (\alpha(-K_{X}) + \beta E).\Gamma = \beta$.  So $\widetilde{E^+} - \beta E = \alpha(-K_{X})$ is Cartier and since the index of $X$ is 1 by assumption, $\alpha \in \Z$.
\end{proof}

\subsection{Existence of Cases}
With the equations and bounds above, we now can write a computer program to numerically classify links (\ref{Firstfig}) with $\phi$ an $E1$-contraction and $\phi^+$ of type $E2, E3/E4$ or $E5$.  In our program we loop through the values of $-K_X^3, r,d,g,\alpha^+$ and $\beta^+$.  The values for all missing variables are then completely determined by the equation in the first section.  The Diophantine equations now used are the ones in (\ref{E1Eequations}).  All the bounds of $(-K_X)^3, r, d,$ and $g$ still apply in the $E1$ case, and we use (\ref{E2alphabounds}) and (\ref{E2alphaIsInt}) to loop through $\alpha^+$ and $\beta^+$, which then completely determine both $\alpha$ and $\beta$.

To exhibit some of these cases geometrically, we use nonsingular del Pezzo surfaces $S_k$ of degree $k = 3,4,5$.  These surfaces are blowups of $\P^2$ at $9-k$ general points $P_1, \ldots, P_{9-k}$.  Denote by $l$ the class of the pullback of a line on $\P^2$ to $S_k$.  denote by $e_i$ the exceptional curves on $S_k$ of the blowup $S_k \rightarrow \P^2$.  Then $l$ and the $e_i$ generate $\Pic(S)$ and there is a nonsingular curve linearly equivalent to a divisor $D$ if and only if $D.L \ge 0$ for every line $L$ on $S_k$ and $D^2 = 0$.  The lines on $S_k$ are the $e_i$, the strict transforms $f_{i,j} \sim l - e_i - e_j$ of lines in $\P^2$ passing through two of the blown up points $P_i$ and $P_j$, and the strict transforms of conics in $\P^2$ passing through five of the $P-I$.  In the case $S_4$, the strict transform of the unique conic through $P_1, \ldots, P_5$ will be denoted by $g \sim 2l - \sum_{i = 1}^5 e_i$.  On $S_3, g_j \sim 2l - \sum_{i \ne j} e_i$ will denote the strict transform of the conic passing through all $P_i$ with $i \ne j$.

\subsubsection{E1-E2}
 Note that $r^+$, the index of the base of the $E2$-contraction, must be $1$.  If $r^+ = 2$, then since $-K_{X^+} = -K_{Y^+} +2E^+$ in this case, the index of $X$ would be 2, which we already ruled out (assumption vi).  If $Y^+$ had index 3 or 4, then $X^+$ would be Fano.

We have to be careful when applying our check on the preservation of Hodge numbers since only for the smooth $E2$ type contraction can we check that $h^{1,2}$ is preserved under a flop.  Since blowing up a nonsingular point does not change the Hodge number $h^{1,2}$, we now have to check that
$$
h^{1,2}(Y) + g = h^{1,2}(Y^+)
$$
After running the program and performing the aforementioned checks, there are only three solutions in the $E1-E2$ case.  These can be found in (\ref{E1E2table}) in the Tables section of this paper.  All three of these solutions are well known examples of smooth weak Fano threefolds with Picard number 2 and their existence is described in (\cite{Tak89}).  While their existence may have been previously known, the new result here is that these are the only three examples of contractions of type $E1-E2$.

\subsubsection{E1-E3/E4}
Numerically, the cases $E1-E3$ and $E1-E4$ are equivalent and so we will treat them as such.  In the examples below, the link found is of type $E1-E3$. We have not found any links of type $E1-E4$.  Since the Fano threefold $Y^+$ now has a singular point, we can not use any of the Hodge number checks as before.  Running a computer program gives seven possible numerical cases which is Table 5.3.  We proceed case by case:

\textit{Cases 1 and 2:} \label{e1e3no12}The same exact argument as in Proposition \ref{shtest1} can be applied to eliminate these cases.

\textit{Case 3}:
\label{e1e3no3}
To eliminate this case, consider the short exact sequence:
$$0 \rightarrow \O_X(-K_X - E) \rightarrow \O_X(-K_X) \rightarrow \O_X(-K_X|_E ) \rightarrow 0$$
Then since $\alpha = -\beta = 1,\, \widetilde{E^+} = -K_X - E$.  So the long exact sequence of cohomology gives:
$$0 \rightarrow H^0(X,\O_X(-K_X - E)) \rightarrow H^0(X,\O_X(-K_X)) \rightarrow H^0(X,\O_X(-K_X|_E)) \rightarrow \ldots$$
Since we are in the $E3/E4$ case, $\O_X(-K_X|_E) \cong \O_E(1)$ thus $h^0(X,\O_X(-K_X|_E)) = 4$.  Since $h^0(X,\O_X(-K_X)) = \frac{-K_X^3}{2} + 3 = 6$, by exactness we must have $$h^0(X,\O_X(-K_X - E)) = h^0(X,\O_X(\widetilde{E^+})) \ge 2$$ which is a contradiction. \\

\textit{Case 4}:
\label{e1e3no4}
By exercise 5.4.8 of \cite{Ha77}, on any $S_3 \subset \P^3$ there is a nonsingular curve $C$ linearly equivalent to $10l - 4e_1 - 4e_2 - 4e_3 - 4e_4 - 3e_5 - 3e_6$. The curve $C$ has genus 6 and degree 8. Let $X$ be the blowup of $C \subset Y$ and let $\widetilde{S}$ denote the strict transform of $S$.  Suppose a curve $\Gamma \subset X$ satisfies $-K_X.\Gamma \le 0$. Since $-K_X = \widetilde{S} + H$, we see that $\widetilde{S}.\Gamma < 0$, so $\Gamma \subset \widetilde{S} \cong S$. Writing $-K_X = 4H + E$ and $\Gamma \sim al - \sum b_ie_i$ in $\widetilde{S} \cong S$, we obtain $4 \cdot \deg_{\P^3}(\Gamma) \le \Gamma.C$, which simplifies to $$2a \le b_5 + b_6$$ The only curves which satisfy this inequality are $e_1, e_2,e_3,e_4$ and $f_{5,6}$, where $f_{5,6}$ is the strict transform of the line through $P_5$ and $P_6$. For each of these curves, equality is attained, so $X$ is a weak Fano threefold with five flopping curves $e_1,e_2,e_3,e_4$ and $f_{5,6}$. Using $$0 \rightarrow \mathcal{N}_{\Gamma/\widetilde{S}} \rightarrow \mathcal{N}_{\Gamma/X} \rightarrow \mathcal{N}_{\widetilde{S}/X}|_{\Gamma} \rightarrow 0$$we see that the normal bundle of each flopping curve is $\O(1) \oplus \O(1)$, so the flop $X \stackrel{\chi}{\dashrightarrow} X^+$ is an Atiyah flop. $\chi(\widetilde{S})$ is isomorphic to $\P^1 \times \P^1$ since it is the result of contracting the flopping curves on $\widetilde{S}$. The sheaf associated to $(\chi(\widetilde{S}))^2$, or equivalently the normal bundle of $\chi(\widetilde{S})$, is then $\O_{\P^1 \times \P^1}(a,b)$ for some integers $a$ and $b$. Choosing a fiber $f$ in $\chi(\widetilde{S})$ not passing though any of the points of intersection of the flopping curves with $\chi(\widetilde{S})$, we see that $a = f.\chi(\widetilde{S})$.  Since the flop $\chi$ restricted to $\widetilde{S}$ is just the blow-down of the five flopping curves,  $f.\chi(\widetilde{S})$ is preserved under the flop so to compute $a$ we will instead compute $\widetilde{f}.\widetilde{S}$.  Using the fact that $\widetilde{S} = 3H - E$ and $\widetilde{f} \sim l - e_5$, we can compute the intersection inside of the cubic surface $\widetilde{S}$.  Since $H|_{\widetilde{S}} \sim 3 l - \sum e_i$ and $E|_{\widetilde{S}} = C \sim 10l - 4e_1 - 4e_2 - 4e_3 - 4e_4 - 3e_5 - 3e_6$, we obtain $a = f.\chi(\widetilde{S}) = \widetilde{f}.\widetilde{S} = -1$.  A similar calculation shows that $b = -1$ and thus the normal bundle of $\chi(\widetilde{S})$ is $\O_{\P^1 \times \P^1} (-1,-1)$, so $X^+$ has an $E3$-contraction with exceptional divisor $\chi(\widetilde{S})$. \\

\textit{Case 5}:
\label{e1e3no5}
Since $-K_Y^3 = 54$, $Y \cong Q$, a smooth quadric threefold in $\P^4$.  Let $S \in |2H|\subset Y$ be a smooth surface.  $S$ is the complete intersection of two quadrics in $\P^4$ so is del Pezzo of degree $K_S^2 = 4$. By exercise 5.4.8 of \cite{Ha77}, there is a nonsingular curve $C$ with degree 8 and genus 4 linearly equivalent to $5l - 2e_1 - 2e_2 - e_3 - e_4 - e_5$ on $S$.  We proceed using the techniques and notation of case 4 above. We obtain $$2a \le b_4 + b_5$$ so $X$ is a weak Fano threefold with four flopping curves $e_1,e_2,e_3$ and $f_{4,5}$. The flop $X \stackrel{\chi}{\dashrightarrow} X^+$ is an Atiyah flop and $X^+$ has an $E3$-contraction with exceptional divisor $\chi(\widetilde{S})$. \\

\textit{Case 6}:
\label{e1e3no6}
Since $-K_Y^3 = 32$, $Y \cong Q_1 \cap Q_2 \subset \P^5$, the smooth intersection of two quadrics.  Let $S \in |H|\subset Y$ be a smooth surface.  $S$ is the complete intersection of two quadrics in $\P^4$ so is del Pezzo of degree $K_S^2 = 4$. By exercise 5.4.8 of \cite{Ha77}, there is a nonsingular curve $C$ with degree 4 and genus 0 linearly equivalent to $4l - 2e_1 - 2e_2 - 2e_3 - e_4 - e_5$ on $S$. We proceed using the techniques and notation of case 4 above. We obtain $$2a \le b_4 + b_5$$ so $X$ is a weak Fano threefold with four flopping curves $e_1,e_2,e_3$ and $f_{4,5}$. The flop $X \stackrel{\chi}{\dashrightarrow} X^+$ is an Atiyah flop and $X^+$ has an $E3$-contraction with exceptional divisor $\chi(\widetilde{S})$. \\

\textit{Case 7}:
\label{e1e3no7}
Since $-K_Y^3 = 40$, $Y \subset \P^6$ is a section of the Grassmannian Gr(2,5) $\subset \P^9$ by a subspace of codimension 3.  Let $S \in |H|\subset Y$ be a smooth surface.  $S$ is del Pezzo of degree $K_S^2 = 5$.  By exercise 5.4.8 of \cite{Ha77}, there is a nonsingular curve $C$ with degree 6 and genus 1 linearly equivalent to $4l - 2e_1 - 2e_2 - e_3 - e_4$ on $S$. We proceed using the techniques and notation of case 4 above. We obtain $$2a \le b_3 + b_4$$ so $X$ is a weak Fano threefold with three flopping curves $e_1,e_2$ and $f_{3,4}$. The flop $X \stackrel{\chi}{\dashrightarrow} X^+$ is an Atiyah flop and $X^+$ has an $E3$-contraction with exceptional divisor $\chi(\widetilde{S})$. \\

\subsubsection{E1-E5}
For the $E1-E5$ case, the table for all the numerical possibilities are listed on Table (\ref{E1E5table}).  We proceed case by case as before:\\

\textit{Case 1}:
\label{e1e5no1}
On $X$, $|-K_X|$ is base point free if and only if the scheme-theoretic base locus of $|-K_Y - C| = |H - C|$ on $Y$ is exactly the curve $C$. But $C$ is a plane elliptic curve and $Y$ is an intersection of quadrics, so this is impossible.

\textit{Case 2}:
\label{e1e5no2}
If this link exists, then the pushdown to $Y$ of a general element of $|-K_X|$ is a smooth $K3$ surface $S \in |-K_Y|$ containing $C$.  Such an $S$ is a quadric section of $Y \cong V_5 \subset \P^9$ and is therefore an intersection of quadrics.  However, by Theorm 1.1 of \cite{Kn02}, a smooth $K3$ surface of degree 10 in $\P^6$ containing a curve of degree 13 and genus 9 cannot be an intersection of quadrics.

\textit{Case 3}:
By exercise 5.4.8 of \cite{Ha77}, on a nonsingular cubic surface $S \subset Y \cong \P^3$, there is a nonsingular curve $C$ in the class $7l - 2e_1 - 2e_2 - 2e_3 - 2e_4 - 2e_5 - 2e_6$ with degree 9 and genus 9. Let $X$ be the blowup of $C \subset Y$ and let $\widetilde{S}$ denote the strict transform of $S$.  Suppose $\Gamma \subset X$ satisfies $-K_X.\Gamma \le 0$. As in E1-E3/E4 Case 4, we see that $\Gamma \subset \widetilde{S} \cong S$. Writing $-K_X = 4H + E$ and $\Gamma \sim al - \sum b_ie_i$ in $\widetilde{S} \cong S$, we obtain $4 \cdot \deg_{\P^3}(\Gamma) \le \Gamma.C$, which simplifies to $$5a \le 2\sum b_i$$ The only curves which satisfy this inequality are $g_1, \ldots,g_6$ which are the strict transform of conics in $\P^2$ passing through 5 of $P_1, \ldots, P_6$. For each of these curves, equality is attained, so $X$ is a weak Fano threefold with six flopping curves $g_1, \ldots,g_6$. Using $$0 \rightarrow \mathcal{N}_{\Gamma/\widetilde{S}} \rightarrow \mathcal{N}_{\Gamma/X} \rightarrow \mathcal{N}_{\widetilde{S}/X}|_{\Gamma} \rightarrow 0$$we see that the normal bundle of each flopping curve is $\O(1) \oplus \O(1)$, so the flop $X \stackrel{\chi}{\dashrightarrow} X^+$ is an Atiyah flop. $\chi(\widetilde{S})$ is isomorphic to $\P^2$ since it is the result of contracting the flopping curves on $\widetilde{S}$. The normal bundle of $\chi(\widetilde{S})$ is $\O_{\P^2}(a)$ for some integer $a$. Choosing a line $f$ in $\chi(\widetilde{S}) \cong \P^2$ not passing though any of the points of intersection of the flopping curves with $\chi(\widetilde{S})$, we see that $a = f.\chi(\widetilde{S})= \widetilde{f}.\widetilde{S}.$  Using the fact that $\widetilde{S} = 3H - E$ and $\widetilde{f} \sim 5l - 2 \sum e_i$, we can compute the intersection inside of the cubic surface $\widetilde{S}$.  Since $H|_{\widetilde{S}} \sim 3 l - \sum e_i$ and $E|_{\widetilde{S}} = C \sim 7l - 2 \sum e_i$, we obtain $a = f.\chi(\widetilde{S}) = \widetilde{f}.\widetilde{S} = -2$. The normal bundle of $\chi(\widetilde{S})$ is $\O_{\P^2}(-2)$, so $X^+$ has an $E5$-contraction with exceptional divisor $\chi(\widetilde{S})$. \\

\textit{Case 4}:
Since $-K_Y^3 = 24$, $Y \subset \P^4$ is a smooth cubic.  Let $S \in |H|\subset Y$ be a smooth surface.  $S$ is del Pezzo of degree $K_S^2 = 3$. By exercise 5.4.8 of \cite{Ha77}, on $S$ there is a nonsingular curve $C$ in the class $5l - 2e_1 - 2e_2 - 2e_3 - 2e_4 - 2e_5 - 2e_6$ with degree 3 and genus 0. We proceed using the techniques and notation of case 3 above. We obtain $$a \le 0$$ so $X$ is a weak Fano threefold with six flopping curves $e_1,\ldots,e_6$. The flop $X \stackrel{\chi}{\dashrightarrow} X^+$ is an Atiyah flop and $X^+$ has an $E5$-contraction with exceptional divisor $\chi(\widetilde{S})$. \\

\textit{Case 5}:
Since $-K_Y^3 = 54$, $Y \cong Q$, a smooth quadric threefold in $\P^4$.  Let $S \in |2H|\subset Y$ be a smooth surface.  $S$ is the complete intersection of two quadrics in $\P^4$ so is del Pezzo of degree $K_S^2 = 4$. By exercise 5.4.8 of \cite{Ha77}, on $S$ there is a nonsingular curve $C$ in the class $7l - 3e_1 - 3e_2 - 2e_3 - 2e_4 - 2e_5$ with degree $9$ and genus $6$.  We proceed using the techniques and notation of case 3 above. We obtain $$2a \le b_3+b_4 + b_5$$ so $X$ is a weak Fano threefold with five flopping curves $e_1,e_2,f_{3,4},f_{3,5}$ and $f_{4,5}$, where $f_{i,j}$ denotes the pullback of a line through the points $P_i$ and $P_j$. The flop $X \stackrel{\chi}{\dashrightarrow} X^+$ is an Atiyah flop and $X^+$ has an $E5$-contraction with exceptional divisor $\chi(\widetilde{S})$. \\

\textit{Case 6}:
Since $-K_Y^3 = 32$, $Y \cong Q_1 \cap Q_2 \subset \P^5$, the smooth intersection of two quadrics.  Let $S \in |H|\subset Y$ be a smooth surface.  $S$ is the complete intersection of two quadrics in $\P^4$ so is del Pezzo of degree $K_S^2 = 4$. By exercise 5.4.8 of \cite{Ha77}, on $S$ there is a nonsingular curve $C$ in the class $5l - 2e_1 - 2e_2 - 2e_3 - 2e_4 - 2e_5$ with degree $5$ and genus $1$.  We proceed using the techniques and notation of case 3 above. We obtain $$a \le 0$$ so $X$ is a weak Fano threefold with five flopping curves $e_1,\ldots e_5$. The flop $X \stackrel{\chi}{\dashrightarrow} X^+$ is an Atiyah flop and $X^+$ has an $E5$-contraction with exceptional divisor $\chi(\widetilde{S})$. \\

\textit{Case 7}:
Since $-K_Y^3 = 40$, $Y \subset \P^6$ is a section of the Grassmannian Gr(2,5) $\subset \P^9$ by a subspace of codimension 3.  Let $S \in |H|\subset Y$ be a smooth surface.  $S$ is del Pezzo of degree $K_S^2 = 5$.  By exercise 5.4.8 of \cite{Ha77}, on $S$ there is a nonsingular curve $C$ in the class $5l - 2e_1 - 2e_2 - 2e_3 - 2e_4$ with degree $7$ and genus $2$.  We proceed using the techniques and notation of case 3 above. We obtain $$a \le 0$$ so $X$ is a weak Fano threefold with four flopping curves $e_1, \ldots,e_4$. The flop $X \stackrel{\chi}{\dashrightarrow} X^+$ is an Atiyah flop and $X^+$ has an $E5$-contraction with exceptional divisor $\chi(\widetilde{S})$. \\

\section{Non E1-E* Cases  (*=2,3,4,5)}
In this section we assume both $\phi: X \rightarrow Y$ and $\phi^+: X^+ \rightarrow Y^+$ are both either $E2, E3/E4$ or $E5$ contractions.  The cases $E3$ and $E4$ are numerically equivalent, and thus we do not distinguish between them in what follows.
We start the classification of these remaining divisorial type contractions with the following Lemma.  By Proposition \ref{E2alphaIsInt}, we know that $\alpha, \alpha^+,\beta,$ and $\beta^+$ are integers.  We immediately obtain:

\begin{lemma}
\label{alphabetaE2E5}
If $X \stackrel{\phi}{\longrightarrow} Y$ and $X^+ \stackrel{\phi^+}{\longrightarrow} Y^+$ are both contractions of type either $E2, E3/E4$ or $E5$,  then $\alpha = \alpha^+$ and $\beta = \beta^+ = -1$.
\end{lemma}
\begin{proof}
Since $\beta$ and $\beta^+$ are both negative integers and $\beta\beta^+ = 1$ by (\ref{firstchecks}), we must have $\beta = \beta^+ = -1$. Since $\alpha + \beta\alpha^+ = 0$, we get $\alpha = \alpha^+$.
\end{proof}

We now proceed to show which combinations of symmetric $E2-E5$ contractions can exist. Note first that by Theorem \ref{FirstMainThm}, a simple calculation shows that $K_X.E^2 = 2$ for any contraction of type $E2, E3/E4$, or $E5$. Then
$$
\begin{array}{cl}
    2   & = K_{X^+}.{E^+}^2 \\
        & = K_X.\widetilde{E^+}^2 \\
        & = K_X.(-\alpha K_X - E)^2 \\
        & = \alpha^2 K_X^3 + 2\alpha K_X^2.E + K_X.E^2 \\
        & = \alpha(\alpha K_X^3 + 2 K_X^2.E)+ 2.
\end{array}
$$

Since $\alpha \neq 0$ by (\ref{Eplusequation}), we obtain:

$$ -K_X^3 = \frac{2K_X^2.E}{\alpha}.$$

By symmetry, this formula also holds for $\phi^+$. That is,

$$ -K_{X^+}^3 = \frac{2K_{X^+}^2.E^+}{\alpha^+}.$$

Recalling that $-K_X^3 = -K_{X^+}^3$ and $\alpha = \alpha^+$, we have $K_X^2.E = K_{X^+}^2.E^+$. Also, $K_X^2.E = 4,2$ or $1$ for contractions of type $E2,E3/E4$ and $E5$ respectively.  Using Theorem \ref{FirstMainThm}, this proves the following:

\begin{theorem}
\label{isabellaebonyemmachloespunkers}
If $X \stackrel{\phi}{\rightarrow} Y$ is of type $E2-E5$, then $X^+ \stackrel{\phi^+}{\rightarrow} Y^+$ is of the same type. The possible values of $(-K_X^3, \alpha)$ are found in Tables \ref{E2E2table}, \ref{E3E3table}, and \ref{E5E5table}.
\end{theorem}

To calculate the remaining the numbers on our tables (\ref{E2E2table}),(\ref{E3E3table}),(\ref{E5E5table}), we note the following formulas for $-K_Y^3$, again using Theorem \ref{FirstMainThm}:
$$\begin{array}{cl}
-K_Y^3 = -K_X^3 + 8, & \phi: X \rightarrow Y \text{ an }E2 \text{ contraction};\\
-K_Y^3 = -K_X^3 + 2, & \phi: X \rightarrow Y \text{ an }E3/E4 \text{ contraction} ;\\
-K_Y^3 = -K_X^3 + \frac{1}{2}, & \phi: X \rightarrow Y \text{ an }E5 \text{ contraction}.\\
\end{array}
$$

\subsection{E2-E2} In Table \ref{E2E2table}, when both contractions are of type $E2$, there are three numerical cases.  One case has been shown to exist by Takeuchi, one shown not to exist by A. Kaloghiros, and one case remains open.

\subsection{E3/E4-E3/E4}
In Table \ref{E3E3table}, when both contractions are of type $E3/E4$, both cases have been previously shown to exist. We make one remark regarding the second numerical case.
\begin{remark}
\textit{Case 2}:
Consider a quartic threefold $Y \subset \P^4$ with a single ordinary double point $P \in Y$. We focus on the birational involution $\tau$ on $Y$ corresponding to projection from $P$. A general line in $\P^4$ through $P$ intersects $Y$ at two additional points and $\tau$ interchanges those two points. We show $\tau$ induces a Sarkisov link of type II.

The indeterminacy of the projection $\pi_P: Y \rightarrow \P^3$ from $P$ is resolved by blowing up $P$. We obtain a smooth threefold $X$ and double cover $\psi: X \rightarrow \P^3$ which coincides with the anticanonical morphism of $X$. A general $Y$ contains 24 lines through $P$, so $X$ is weak Fano with 24 flopping curves. The flop $\chi$ is the extension $\tau_X$ of $\tau$ to $X$. Blowing up the flopping curves, we obtain a variety $Z$ on which the extension $\tau_Z$ of $\tau$ is regular. Thus $\chi$ is an Atiyah flop with the corresponding commutative diagram.\\

$$\xymatrix{Z \ar[d] \ar[rr]^{\tau_Z} & & Z \ar[d] \\ X \ar@{-->}[rr]^{\chi = \tau_X} \ar[dd]_{\phi} \ar[dr]^{\psi}& & X \ar[dd]^{\phi} \ar[dl]_{\psi} \\
           & \P^3 & \\ Y \ar@{-->}[ur]^{\pi_P} \ar@{-->}[rr]^{\tau} & & Y \ar@{-->}[ul]_{\pi_P}}$$
\end{remark}

\vspace{1cm}

\subsection{E5-E5} The existence of the one numerical case of type $E5 - E5$ in Table \ref{E5E5table} remains open.

%%%%%%%%%%%%%%%%%%%%%%%%Tables%%%%%%%%%%%%%%%%%%%%
\newpage
\section{Tables}
\label{Tables}
\subsection{E1 - E1}
The following table is a list of all the numerical possibilities for the $E1 - E1$ case, where $r$ and $r^+$ are the index of $Y$ and $Y^+$ in (\ref{Firstfig}) respectively.  All other notations can be found in (\ref{Eplusequation}) and (\ref{Eequation}).  Due to space constraints, missing from the table are the values of $\alpha^+$ and $\beta^+$.  Those can be determined from the given values of $\alpha$ and $\beta$ using (\ref{firstchecks}). There are 111 entries on the table: 13 proven not to exist, 54 proven to exist, and the existence of the remaining cases are unknown.  These cases are denoted by the standard scientific notations ``x", ``\,:)\,", and ``?" respectively.

\label{E1E1table}
\begin{table}[h]
\caption{E1-E1}
\begin{center}
\begin{tabular}{|c|c|c|c|c|c|c|c|c|c|c|c|c|c|c|}
\hline
\textit{No.} & $-K_X^3$ & $-K_Y^3$ & $-K_{Y^+}^3$ & $\alpha$ & $\beta$ & $r$ & $d$ & $g$ & $r^+$ & $d^+$ & $g^+$ & $e/r^3$ & Exist? & Ref \\ \hline \hline
$\textit{1.}$ & 2 & 6 & 6 & 3 & -1 & 1 & 1 & 0 & 1 & 1 & 0 & 47 & :) & [Isk78] \\ \hline
$\textit{2.}$ & 2 & 8 & 8 & 4 & -1 & 1 & 2 & 0 & 1 & 2 & 0 & 88 & ? &  \\ \hline
$\textit{3.}$ & 2 & 10 & 10 & 5 & -1 & 1 & 3 & 0 & 1 & 3 & 0 & 153 & :) & \cite{ACM11} \\ \hline
$\textit{4.}$ & 2 & 12 & 12 & 6 & -1 & 1 & 4 & 0 & 1 & 4 & 0 & 248 & :) &  \cite{ACM11}\\ \hline
$\textit{5.}$ & 2 & 14 & 14 & 7 & -1 & 1 & 5 & 0 & 1 & 5 & 0 & 379 & :) & \cite{ACM11} \\ \hline
$\textit{6.}$ & 2 & 16 & 16 & 8 & -1 & 1 & 6 & 0 & 1 & 6 & 0 & 552 & :) &  \cite{ACM11}\\ \hline
$\textit{7.}$ & 2 & 18 & 18 & 9 & -1 & 1 & 7 & 0 & 1 & 7 & 0 & 773 & :) &  \cite{ACM11}\\ \hline
$\textit{8.}$ & 2 & 22 & 22 & 11 & -1 & 1 & 9 & 0 & 1 & 9 & 0 & 1383 & :) & \cite{ACM11} \\ \hline
$\textit{9.}$ & 2 & 8 & 8 & 3 & -1 & 1 & 3 & 1 & 1 & 3 & 1 & 21 & x & (\ref{e1e1case9}) \\ \hline
$\textit{10.}$ & 2 & 10 & 10 & 4 & -1 & 1 & 4 & 1 & 1 & 4 & 1 & 56 & ? &  \\ \hline
$\textit{11.}$ & 2 & 12 & 12 & 5 & -1 & 1 & 5 & 1 & 1 & 5 & 1 & 115 & :) &  \cite{ACM11} \\ \hline
$\textit{12.}$ & 2 & 14 & 14 & 6 & -1 & 1 & 6 & 1 & 1 & 6 & 1 & 204 & :) & \cite{ACM11} \\ \hline
$\textit{13.}$ & 2 & 16 & 16 & 7 & -1 & 1 & 7 & 1 & 1 & 7 & 1 & 329 & :) &  \cite{ACM11} \\ \hline
$\textit{14.}$ & 2 & 18 & 18 & 8 & -1 & 1 & 8 & 1 & 1 & 8 & 1 & 496 & :) &   \cite{ACM11} \\ \hline
$\textit{15.}$ & 2 & 22 & 22 & 10 & -1 & 1 & 10 & 1 & 1 & 10 & 1 & 980 & :) &  \cite{ACM11} \\ \hline
$\textit{16.}$ & 2 & 12 & 12 & 4 & -1 & 1 & 6 & 2 & 1 & 6 & 2 & 24 & x &  \cite{ACM11} \\ \hline
$\textit{17.}$ & 2 & 14 & 14 & 5 & -1 & 1 & 7 & 2 & 1 & 7 & 2 & 77 & x & \cite{ACM11} \\ \hline
$\textit{18.}$ & 2 & 16 & 16 & 6 & -1 & 1 & 8 & 2 & 1 & 8 & 2 & 160 & :) &  \cite{ACM11} \\ \hline
$\textit{19.}$ & 2 & 18 & 18 & 7 & -1 & 1 & 9 & 2 & 1 & 9 & 2 & 279 & :) &  \cite{ACM11} \\ \hline
$\textit{20.}$ & 2 & 22 & 22 & 9 & -1 & 1 & 11 & 2 & 1 & 11 & 2 & 649 & :) &   \cite{ACM11} \\ \hline
$\textit{21.}$ & 2 & 16 & 16 & 5 & -1 & 1 & 9 & 3 & 1 & 9 & 3 & 39 & x &   \cite{ACM11}  \\ \hline
$\textit{22.}$ & 2 & 18 & 18 & 6 & -1 & 1 & 10 & 3 & 1 & 10 & 3 & 116 & x &  \cite{ACM11} \\ \hline
$\textit{23.}$ & 2 & 22 & 22 & 8 & -1 & 1 & 12 & 3 & 1 & 12 & 3 & 384 & :) &  \cite{ACM11} \\ \hline
$\textit{24.}$ & 2 & 18 & 18 & 5 & -1 & 1 & 11 & 4 & 1 & 11 & 4 & 1 & x &  \cite{ACM11} \\ \hline
$\textit{25.}$ & 2 & 22 & 22 & 7 & -1 & 1 & 13 & 4 & 1 & 13 & 4 & 179 & x &  \cite{ACM11} \\ \hline
$\textit{26.}$ & 2 & 22 & 22 & 6 & -1 & 1 & 14 & 5 & 1 & 14 & 5 & 28 & x &  \cite{ACM11} \\ \hline
$\textit{27.}$ & 2 & 8 & 8 & 4 & -1 & 2 & 1 & 0 & 2 & 1 & 0 & 11 & x & (\ref{shtest1}) \\ \hline
$\textit{28.}$ & 2 & 16 & 16 & 8 & -1 & 2 & 3 & 0 & 2 & 3 & 0 & 69 & ? &  \\ \hline
$\textit{29.}$ & 2 & 24 & 24 & 12 & -1 & 2 & 5 & 0 & 2 & 5 & 0 & 223 & :) & (\ref{MplusFmethod}) \\ \hline
$\textit{30.}$ & 2 & 32 & 32 & 16 & -1 & 2 & 7 & 0 & 2 & 7 & 0 & 521 & :) & \cite{ACM11} \\ \hline
\end{tabular}
\label{tb:E1E1Table1}
\end{center}
\end{table}

\begin{table}
\caption{E1-E1 (continued)}
\begin{center}
\begin{tabular}{|c|c|c|c|c|c|c|c|c|c|c|c|c|c|c|}
\hline
\textit{No.} & $-K_X^3$ & $-K_Y^3$ & $-K_{Y^+}^3$ & $\alpha$ & $\beta$ & $r$ & $d$ & $g$ & $r^+$ & $d^+$ & $g^+$ & $e/r^3$ & Exist? & Ref \\ \hline \hline
$\textit{31.}$ & 2 & 40 & 40 & 20 & -1 & 2 & 9 & 0 & 2 & 9 & 0 & 1011 & :) &  \cite{ACM11}  \\ \hline
$\textit{32.}$ & 2 & 16 & 16 & 6 & -1 & 2 & 4 & 2 & 2 & 4 & 2 & 20 & x & (\ref{shtest1}) \\ \hline
$\textit{33.}$ & 2 & 24 & 24 & 10 & -1 & 2 & 6 & 2 & 2 & 6 & 2 & 114 & :) & (\ref{MplusFmethod})  \\ \hline
$\textit{34.}$ & 2 & 32 & 32 & 14 & -1 & 2 & 8 & 2 & 2 & 8 & 2 & 328 & :) & \cite{ACM11} \\ \hline
$\textit{35.}$ & 2 & 40 & 40 & 18 & -1 & 2 & 10 & 2 & 2 & 10 & 2 & 710 & :) &  \cite{ACM11} \\ \hline
$\textit{36.}$ & 2 & 24 & 24 & 8 & -1 & 2 & 7 & 4 & 2 & 7 & 4 & 41 & x & (\ref{shtest1}) \\ \hline
$\textit{37.}$ & 2 & 32 & 32 & 12 & -1 & 2 & 9 & 4 & 2 & 9 & 4 & 183 & :) & \cite{ACM11}  \\ \hline
$\textit{38.}$ & 2 & 40 & 40 & 16 & -1 & 2 & 11 & 4 & 2 & 11 & 4 & 469 & :) &  \cite{ACM11}  \\ \hline
$\textit{39.}$ & 2 & 32 & 32 & 10 & -1 & 2 & 10 & 6 & 2 & 10 & 6 & 80 & ? &  \\ \hline
$\textit{40.}$ & 2 & 40 & 40 & 14 & -1 & 2 & 12 & 6 & 2 & 12 & 6 & 282 & :) &  \cite{ACM11}  \\ \hline
$\textit{41.}$ & 2 & 32 & 32 & 8 & -1 & 2 & 11 & 8 & 2 & 11 & 8 & 13 & :) & \cite{ACM11}  \\ \hline
$\textit{42.}$ & 2 & 40 & 40 & 12 & -1 & 2 & 13 & 8 & 2 & 13 & 8 & 143 & :) &  \cite{ACM11}  \\ \hline
$\textit{43.}$ & 2 & 40 & 40 & 10 & -1 & 2 & 14 & 10 & 2 & 14 & 10 & 46 & x &  \cite{ACM11} \\ \hline
$\textit{44.}$ & 2 & 54 & 54 & 25 & -1 & 3 & 9 & 2 & 3 & 9 & 2 & 571 & :) & \cite{ACM11} \\ \hline
$\textit{45.}$ & 2 & 54 & 54 & 22 & -1 & 3 & 10 & 5 & 3 & 10 & 5 & 372 & :) & \cite{ACM11} \\ \hline
$\textit{46.}$ & 2 & 54 & 54 & 19 & -1 & 3 & 11 & 8 & 3 & 11 & 8 & 221 & :) & \cite{ACM11} \\ \hline
$\textit{47.}$ & 2 & 54 & 54 & 16 & -1 & 3 & 12 & 11 & 3 & 12 & 11 & 112 & ? & \\ \hline
$\textit{48.}$ & 2 & 54 & 54 & 13 & -1 & 3 & 13 & 14 & 3 & 13 & 14 & 39 & :) & \cite{ACM11} \\ \hline
$\textit{49.}$ & 2 & 64 & 64 & 30 & -1 & 4 & 8 & 2 & 4 & 8 & 2 & 418 & :) & (\ref{MplusFmethod})  \\ \hline
$\textit{50.}$ & 2 & 64 & 64 & 26 & -1 & 4 & 9 & 6 & 4 & 9 & 6 & 261 & :) & (\ref{MplusFmethod}) \\ \hline
$\textit{51.}$ & 2 & 64 & 64 & 22 & -1 & 4 & 10 & 10 & 4 & 10 & 10 & 146 & :) & (\ref{MplusFmethod}) \\ \hline
$\textit{52.}$ & 2 & 64 & 64 & 18 & -1 & 4 & 11 & 14 & 4 & 11 & 14 & 67 & :) & (\ref{MplusFmethod})  \\ \hline
$\textit{53.}$ & 2 & 64 & 64 & 14 & -1 & 4 & 12 & 18 & 4 & 12 & 18 & 18 & x & (\ref{shtest1}) \\ \hline
$\textit{54.}$ & 4 & 10 & 10 & 2 & -1 & 1 & 2 & 0 & 1 & 2 & 0 & 28 & :) & [Tak89] \\ \hline
$\textit{55.}$ & 4 & 14 & 14 & 3 & -1 & 1 & 4 & 0 & 1 & 4 & 0 & 68 & :) & \cite{ACM11} \\ \hline
$\textit{56.}$ & 4 & 18 & 18 & 4 & -1 & 1 & 6 & 0 & 1 & 6 & 0 & 144 & :) &   \cite{ACM11} \\ \hline
$\textit{57.}$ & 4 & 22 & 22 & 5 & -1 & 1 & 8 & 0 & 1 & 8 & 0 & 268 & :) &  \cite{ACM11} \\ \hline
$\textit{58.}$ & 4 & 12 & 12 & 2 & -1 & 1 & 4 & 1 & 1 & 4 & 1 & 8 & x &   \cite{ACM11} \\ \hline
$\textit{59.}$ & 4 & 16 & 16 & 3 & -1 & 1 & 6 & 1 & 1 & 6 & 1 & 42 & ? &  \\ \hline
$\textit{60.}$ & 4 & 18 & 18 & 3 & -1 & 1 & 8 & 2 & 1 & 8 & 2 & 16 & x &  \cite{ACM11} \\ \hline
$\textit{61.}$ & 4 & 22 & 22 & 4 & -1 & 1 & 10 & 2 & 1 & 10 & 2 & 80 & ? &  \\ \hline
$\textit{62.}$ & 4 & 16 & 10 & 1.5 & -0.5 & 2 & 3 & 1 & 1 & 3 & 1 & 3 & x & (\ref{shtest1}) \\ \hline
$\textit{63.}$ & 4 & 24 & 14 & 2.5 & -0.5 & 2 & 5 & 1 & 1 & 5 & 1 & 25 & :) & [Isk78] \\ \hline
$\textit{64.}$ & 4 & 32 & 18 & 3.5 & -0.5 & 2 & 7 & 1 & 1 & 7 & 1 & 77 & :) & \cite{ACM11} \\ \hline
$\textit{65.}$ & 4 & 40 & 22 & 4.5 & -0.5 & 2 & 9 & 1 & 1 & 9 & 1 & 171 & :) &  \cite{ACM11} \\ \hline
$\textit{66.}$ & 4 & 24 & 24 & 4 & -1 & 2 & 6 & 3 & 2 & 6 & 3 & 6 & x & (\ref{shtest1}) \\ \hline
$\textit{67.}$ & 4 & 32 & 32 & 6 & -1 & 2 & 8 & 3 & 2 & 8 & 3 & 40 & ? &  \\ \hline
$\textit{68.}$ & 4 & 40 & 40 & 8 & -1 & 2 & 10 & 3 & 2 & 10 & 3 & 110 & :) &  \cite{ACM11} \\ \hline
$\textit{69.}$ & 4 & 40 & 40 & 6 & -1 & 2 & 12 & 7 & 2 & 12 & 7 & 18 & :) &  \cite{ACM11} \\ \hline
$\textit{70.}$ & 4 & 54 & 16 & 11/3 & -1/3 & 3 & 9 & 3 & 1 & 5 & 0 & 103 & :) & \cite{ACM11} \\ \hline
\end{tabular}
\label{tb:E1E1Table2}
\end{center}
\end{table}

\begin{table}
\caption{E1-E1 (continued)}
\begin{center}
\begin{tabular}{|c|c|c|c|c|c|c|c|c|c|c|c|c|c|c|}
\hline
\textit{No.} & $-K_X^3$ & $-K_Y^3$ & $-K_{Y^+}^3$ & $\alpha$ & $\beta$ & $r$ & $d$ & $g$ & $r^+$ & $d^+$ & $g^+$ & $e/r^3$ & Exist? & Ref \\ \hline \hline
$\textit{71.}$ & 4 & 54 & 54 & 13 & -1 & 3 & 8 & 0 & 3 & 8 & 0 & 164 & :) &   \cite{ACM11} \\ \hline
$\textit{72.}$ & 4 & 54 & 54 & 10 & -1 & 3 & 10 & 6 & 3 & 10 & 6 & 60 & ? &  \\ \hline
$\textit{73.}$ & 4 & 54 & 54 & 7 & -1 & 3 & 12 & 12 & 3 & 12 & 12 & 4 & x & (\ref{shtest1}) \\ \hline
$\textit{74.}$ & 4 & 64 & 12 & 2.75 & -0.25 & 4 & 9 & 7 & 1 & 3 & 0 & 45 & :) & \small{[Tak89,IP99]} \\ \hline
$\textit{75.}$ & 4 & 64 & 64 & 14 & -1 & 4 & 8 & 3 & 4 & 8 & 3 & 82 & :) & (\ref{MplusFmethod}) \\ \hline
$\textit{76.}$ & 4 & 64 & 64 & 10 & -1 & 4 & 10 & 11 & 4 & 10 & 11 & 20 & :) & \cite{JPR05}  \\ \hline
$\textit{77.}$ & 6 & 10 & 10 & 1 & -1 & 1 & 1 & 0 & 1 & 1 & 0 & 11 & :) & [Isk78] \\ \hline
$\textit{78.}$ & 6 & 16 & 16 & 2 & -1 & 1 & 4 & 0 & 1 & 4 & 0 & 32 & ? &  \\ \hline
$\textit{79.}$ & 6 & 22 & 22 & 3 & -1 & 1 & 7 & 0 & 1 & 7 & 0 & 89 & :) &  \cite{ACM11} \\ \hline
$\textit{80.}$ & 6 & 18 & 18 & 2 & -1 & 1 & 6 & 1 & 1 & 6 & 1 & 12 & ? &  \\ \hline
$\textit{81.}$ & 6 & 40 & 18 & 2.5 & -0.5 & 2 & 9 & 2 & 1 & 5 & 0 & 47 & :) &  \cite{ACM11} \\ \hline
$\textit{82.}$ & 6 & 16 & 16 & 2 & -1 & 2 & 2 & 0 & 2 & 2 & 0 & 4 & x & (\ref{shtest1}) \\ \hline
$\textit{83.}$ & 6 & 40 & 40 & 6 & -1 & 2 & 8 & 0 & 2 & 8 & 0 & 82 & :) &  \cite{ACM11}  \\ \hline
$\textit{84.}$ & 6 & 32 & 32 & 4 & -1 & 2 & 7 & 2 & 2 & 7 & 2 & 17 & :) & \cite{ACM11} \\ \hline
$\textit{85.}$ & 6 & 40 & 40 & 4 & -1 & 2 & 11 & 6 & 2 & 11 & 6 & 1 & x & (\ref{shtest1}) \\ \hline
$\textit{86.}$ & 6 & 54 & 22 & 8/3 & -1/3 & 3 & 8 & 1 & 1 & 8 & 1 & 48 & :) & \cite{ACM11} \\ \hline
$\textit{87.}$ & 6 & 54 & 12 & 5/3 & -1/3 & 3 & 10 & 7 & 1 & 2 & 0 & 14 & :) & [Tak89] \\ \hline
$\textit{88.}$ & 6 & 54 & 54 & 7 & -1 & 3 & 9 & 4 & 3 & 9 & 4 & 31 & :) & \cite{ACM11} \\ \hline
$\textit{89.}$ & 6 & 64 & 40 & 4.5 & -0.5 & 4 & 8 & 4 & 2 & 10 & 4 & 24 & :) & (\ref{MplusFmethod})  \\ \hline
$\textit{90.}$ & 6 & 64 & 64 & 10 & -1 & 4 & 7 & 0 & 4 & 7 & 0 & 47 & :) & (\ref{MplusFmethod})  \\ \hline
$\textit{91.}$ & 6 & 64 & 64 & 6 & -1 & 4 & 10 & 12 & 4 & 10 & 12 & 2 & x & (\ref{shtest1}) \\ \hline
$\textit{92.}$ & 8 & 14 & 14 & 1 & -1 & 1 & 2 & 0 & 1 & 2 & 0 & 10 & :) & [Tak89] \\ \hline
$\textit{93.}$ & 8 & 22 & 22 & 2 & -1 & 1 & 6 & 0 & 1 & 6 & 0 & 36 & :) &  \cite{ACM11} \\ \hline
$\textit{94.}$ & 8 & 40 & 16 & 1.5 & -0.5 & 2 & 9 & 3 & 1 & 3 & 0 & 12 & :) &  \cite{ACM11}  \\ \hline
$\textit{95.}$ & 8 & 24 & 24 & 2 & -1 & 2 & 4 & 1 & 2 & 4 & 1 & 2 & x & (\ref{shtest1}) \\ \hline
$\textit{96.}$ & 8 & 40 & 40 & 4 & -1 & 2 & 8 & 1 & 2 & 8 & 1 & 28 & :) &  \cite{ACM11}  \\ \hline
$\textit{97.}$ & 8 & 54 & 18 & 5/3 & -1/3 & 3 & 8 & 2 & 1 & 4 & 0 & 20 & :) & \cite{ACM11}  \\ \hline
$\textit{98.}$ & 8 & 64 & 22 & 1.75 & -0.25 & 4 & 7 & 1 & 1 & 7 & 1 & 14 & :) & (\ref{MplusFmethod})  \\ \hline
$\textit{99.}$ & 8 & 64 & 64 & 6 & -1 & 4 & 8 & 5 & 4 & 8 & 5 & 10 & :) & \cite{BL11}  \\ \hline
$\textit{100.}$ & 10 & 18 & 18 & 1 & -1 & 1 & 3 & 0 & 1 & 3 & 0 & 9 & :) &  \cite{ACM11}  \\ \hline
$\textit{101.}$ & 10 & 40 & 22 & 1.5 & -0.5 & 2 & 7 & 0 & 1 & 5 & 0 & 18 & :) &  \cite{ACM11}  \\ \hline
$\textit{102.}$ & 10 & 54 & 54 & 4 & -1 & 3 & 8 & 3 & 3 & 8 & 3 & 8 & ? &   \\ \hline
$\textit{103.}$ & 10 & 64 & 32 & 2.5 & -0.5 & 4 & 7 & 2 & 2 & 5 & 0 & 9 & :) & (\ref{MplusFmethod}) \\ \hline
$\textit{104.}$ & 12 & 22 & 22 & 1 & -1 & 1 & 4 & 0 & 1 & 4 & 0 & 8 & :) &  \cite{ACM11} \\ \hline
$\textit{105.}$ & 12 & 54 & 40 & 7/3 & -2/3 & 3 & 7 & 1 & 2 & 7 & 1 & 7 & :) & \cite{ACM11}  \\ \hline
$\textit{106.}$ & 12 & 64 & 16 & 0.75 & -0.25 & 4 & 7 & 3 & 1 & 1 & 0 & 5 & :) & [Isk78] \\ \hline
$\textit{107.}$ & 14 & 40 & 40 & 2 & -1 & 2 & 6 & 0 & 2 & 6 & 0 & 6 & :) & \cite{ChSh11}  \\ \hline
$\textit{108.}$ & 14 & 54 & 18 & 2/3 & -1/3 & 3 & 7 & 2 & 1 & 1 & 0 & 4 & :) & [Isk78] \\ \hline
$\textit{109.}$ & 16 & 54 & 22 & 2/3 & -1/3 & 3 & 6 & 0 & 1 & 2 & 0 & 4 & :) & [Tak89] \\ \hline
$\textit{110.}$ & 18 & 40 & 22 & 0.5 & -0.5 & 2 & 5 & 0 & 1 & 1 & 0 & 3 & :) & [Isk78] \\ \hline
$\textit{111.}$ & 22 & 64 & 64 & 2 & -1 & 4 & 5 & 0 & 4 & 5 & 0 & 1 & :) & (\ref{ex111}) \\  \hline
\end{tabular}
\label{tb:E1E1Table3}
\end{center}
\end{table}

\newpage
\subsection{E1 - E2}
The following table is a list of all the numerical possibilities for the $E1 - E2$ case, where $r$ is the index of $Y$ in (\ref{Firstfig}).  The index of $Y^+$ is always 1.  All other notations can be found in (\ref{Eplusequation}) and (\ref{Eequation}).  Due to space constraints, missing from the table are the values of $\alpha^+$ and $\beta^+$.  Those can be determined from the given values of $\alpha$ and $\beta$ using (\ref{firstchecks}). There are 3 entries on the table, all known to exist.  Their geometric realizations are shown in Takeuchi's paper (\cite{Tak89}).

\label{E1E2table}
\begin{table}[h]
\caption{E1-E2}
\begin{center}
\begin{tabular}{|c|c|c|c|c|c|c|c|c|c|c|c|}
\hline
\textit{No.} & $-K_X^3$ & $-K_Y^3$ & $-K_{Y^+}^3$ & $\alpha$ & $\beta$ & $r$ & $d$ & $g$ & $e/r^3$ & Exist? & Ref \\ \hline \hline
$\textit{1.}$ & 4 & 40 & 12 & 5/2 & -1/2 & 2 & 12 & 7 & 24 & :) & \cite{Tak89} \\ \hline
$\textit{2.}$ & 6 & 24 & 14 & 3/2 & -1/2 & 2 & 4 & 0 & 16 & :) & \cite{Tak89} \\ \hline
$\textit{3.}$ & 14 & 64 & 22 & 3/4 & -1/4 & 4 & 6 & 0 & 6 & :) & \cite{Tak89} \\ \hline
\end{tabular}
\label{tb:E1E2Table}
\end{center}
\end{table}

\newpage

\subsection{E1 - E3/E4}
The following table is a list of all the numerical possibilities for the $E1 - E3/E4$ case, where $r$ is the index of $Y$ in (\ref{Firstfig}).  The index of $Y^+$ can always be assumed to equal 1.  All other notations can be found in (\ref{Eplusequation}) and (\ref{Eequation}).  Due to space constraints, missing from the table are the values of $\alpha^+$ and $\beta^+$.  Those can be determined from the given values of $\alpha$ and $\beta$ using (\ref{firstchecks}). There are 7 entries on the table, 3 which do not exist and 4 which are geometrically realizable.
\label{E1E3table}
\begin{table}[h]
\caption{E1-E3}
\begin{center}
\begin{tabular}{|c|c|c|c|c|c|c|c|c|c|c|c|}
\hline
\textit{No.} & $-K_X^3$ & $-K_Y^3$ & $-K_{Y^+}^3$ & $\alpha$ & $\beta$ & $r$ & $d$ & $g$ & $e/r^3$ & Exist? & Ref \\ \hline \hline
$\textit{1.}$ & 4 & 24 & 6 & 1.5 & -0.5 & 2 & 6 & 3 & 9 & x & \ref{e1e3no12} \\ \hline
$\textit{2.}$ & 4 & 54 & 6 & 5/3 & -1/3 & 3 & 12 & 12 & 8 & x & \ref{e1e3no12} \\ \hline
$\textit{3.}$ & 6 & 14 & 8 & 1 & -1 & 1 & 4 & 1 & 4 & x & \ref{e1e3no3} \\ \hline
$\textit{4.}$ & 10 & 64 & 12 & 0.75 & -0.25 & 4 & 8 & 6 & 5 & :) & \ref{e1e3no4} \\ \hline
$\textit{5.}$ & 12 & 54 & 14 & 2/3 & -1/3 & 3 & 8 & 4 & 4 & :) & \ref{e1e3no5} \\ \hline
$\textit{6.}$ & 14 & 32 & 16 & 0.5 & -0.5 & 2 & 4 & 0 & 4 & :) & \ref{e1e3no6} \\ \hline
$\textit{7.}$ & 16 & 40 & 18 & 0.5 & -0.5 & 2 & 6 & 1 & 3 & :) & \ref{e1e3no7} \\ \hline
\end{tabular}
\label{tb:E1E3Table}
\end{center}
\end{table}

\newpage
\subsection{E1 - E5}
The following table is a list of all the numerical possibilities for the $E1 - E5$ case, where $r$ is the index of $Y$ in (\ref{Firstfig}). All other notation can be found in (\ref{Eplusequation}) and (\ref{Eequation}).  Due to space constraints, missing from the table are the values of $\alpha^+$ and $\beta^+$.  Those can be determined from the given values of $\alpha$ and $\beta$ using (\ref{firstchecks}). There are 7 entries on the table, 2 known not to exist and 5 which are geometrically realizable.
\label{E1E5table}
\begin{table}[h]
\caption{E1-E5}
\begin{center}
\begin{tabular}{|c|c|c|c|c|c|c|c|c|c|c|c|}
\hline
\textit{No.} & $-K_X^3$ & $-K_Y^3$ & $-K_{Y^+}^3$ & $\alpha$ & $\beta$ & $r$ & $d$ & $g$ & $e/r^3$ & Exist? & Ref \\ \hline \hline
$\textit{1.}$ & 4 & 10 & 9/2 & 1 & -1 & 1 & 3 & 1 & 6 & x & \ref{e1e5no1} \\ \hline
$\textit{2.}$ & 4 & 40 & 9/2 & 1.5 & -0.5 & 2 & 13 & 9 & 1 & x & \ref{e1e5no2} \\ \hline
$\textit{3.}$ & 8 & 64 & 17/2 & 0.75 & -0.25 & 4 & 9 & 9 & 6 & :) & \ref{e1e5no1} \\ \hline
$\textit{4.}$ & 10 & 24 & 21/2 & 0.5 & -0.5 & 2 & 3 & 0 & 6 & :) & \ref{e1e5no1} \\ \hline
$\textit{5.}$ & 10 & 54 & 21/2 & 2/3 & -1/3 & 3 & 9 & 6 & 5 & :) & \ref{e1e5no1} \\ \hline
$\textit{6.}$ & 12 & 32 & 25/2 & 0.5 & -0.5 & 2 & 5 & 1 & 5 & :) & \ref{e1e5no1} \\ \hline
$\textit{7.}$ & 14 & 40 & 29/2 & 0.5 & -0.5 & 2 & 7 & 2 & 4 & :) & \ref{e1e5no1} \\ \hline
\end{tabular}
\label{tb:E1E5Table}
\end{center}
\end{table}

\newpage

\subsection{E2 - E2}
The following table is a list of all the numerical possibilities for the $E2 - E2$ case.  The Fano indices of $Y^+$ and $Y$ can always be assumed to equal 1 as shown in the proof of Proposition \ref{E2alphaIsInt}.  All other notations can be found in (\ref{Eplusequation}) and (\ref{Eequation}).  Due to space constraints, missing from the table are the values of $\alpha^+$ and $\beta^+$ which are equal to $\alpha$ and $\beta$ respectively using (\ref{firstchecks}). There are 3 entries on the table, 1 known not to exist, 1 which is geometrically realizable, and 1 unknown.
\label{E2E2table}
\begin{table}[!h]
\caption{E2-E2}
\begin{center}
\begin{tabular}{|c|c|c|c|c|c|c|c|}
\hline
\textit{No.} & $-K_X^3$ & $-K_{Y}^3$ & $\alpha$ & $\beta$  & $e$ & Exist? & Ref \\ \hline
$\textit{1.}$ & 8 & 16 & 1 & -1 & 12 & :)  & [Tak89] \\ \hline
$\textit{2.}$ & 4 & 12 & 2 & -1 & 30 & x  & [Kal09] \\ \hline
$\textit{3.}$ & 2 & 10 & 4 & -1 & 90 & ? &  \\ \hline
\end{tabular}
\label{tb:E2E2Table}
\end{center}
\end{table}

\clearpage
\newpage

\subsection{E3/4 - E3/4}
The following table is a list of all the numerical possibilities for the $E3/4 - E3/4$ case. All notation can be found in (\ref{Eplusequation}) and (\ref{Eequation}).  Due to space constraints, missing from the table are the values of $\alpha^+$ and $\beta^+$, which are equal to $\alpha$ and $\beta$ respectively using (\ref{firstchecks}). There are 2 entries on the table, both known to exist.
\label{E3E3table}
\begin{table}[h]
\caption{E3/4-E3/4}
\begin{center}
\begin{tabular}{|c|c|c|c|c|c|c|c|}
\hline
\textit{No.} & $-K_X^3$ & $-K_{Y}^3$ & $\alpha$ & $\beta$  & $e$ & Exist? & Ref \\ \hline
$\textit{1.}$ & 4 & 6 & 1 & -1 & 12 & :) & [Kal09] \\ \hline
$\textit{2.}$ & 2 & 4 & 2 & -1 & 24 & :) & [Puk88] \\ \hline
\end{tabular}
\label{tb:E3E3Table}
\end{center}
\end{table}
\clearpage
\newpage

\subsection{E5 - E5}
The following table is a list of the only numerical possibility for the $E5 - E5$ case. All notation can be found in (\ref{Eplusequation}) and (\ref{Eequation}).  Due to space constraints, missing from the table are the values of $\alpha^+$ and $\beta^+$ which are equal to $\alpha$ and $\beta$ respectively using (\ref{firstchecks}).
\label{E5E5table}
\begin{table}[h]
\caption{E5-E5}
\begin{center}
\begin{tabular}{|c|c|c|c|c|c|c|c|}
\hline
\textit{No.} & $-K_X^3$ & $-K_{Y}^3$ & $\alpha$ & $\beta$  & $e$ & Exist? & Ref \\ \hline
$\textit{1.}$ & 2 & 2.5 & 1 & -1 & 15 & ? &  \\ \hline
\end{tabular}
\label{tb:E5E5Table}
\end{center}
\end{table}
\newpage

\newpage

\section*{Acknowledgements}
We would like to thank our advisor, Dr. V.V. Shokurov, for his constant patience, help and suggestions.  We would also like to thank the faculty of the Steklov Institute who were so hospitable during our stay there while working on this paper.  In particular, we would especially like to thank Dr. Y. Prokhorov for his helpful and insightful comments into our research.  We would also like to thank Dr. S. Zucker for his time, advice, and helpfulness in the preparation and presentation of this paper.

%%%%%%%%%%%%%%%%%%%%%%%%%%%%%%%%


\newpage
\begin{thebibliography}{10}
\bibitem[ACM11]{ACM11}M. Arap, J. Cutrone, N. Marshburn. On the Construction of Certain Weak Fano Threefolds of Picard Number Two. \href{http://arxiv.org/pdf/1112.2611v1.pdf}{	arXiv:1112.2611v1 [math.AG]}.
\bibitem[Ba]{Ba} V. Batyrev. Stringy Hodge numbers of varieties with Gorenstein canonical singularities, Integrable systems and algebraic geometry (Kobe/Kyoto, 1997), World Sci. Publ., River Edge, NJ (1998), 1-32.
\bibitem[BL11]{BL11} J. Blanc, S. Lamy.  Weak Fano Threefolds Obtained by Blowing-up a Space Curve and Construction of Sarkisov Links. \href{http://arxiv.org/PS_cache/arxiv/pdf/1106/1106.3716v2.pdf}{arXiv:1106.3716v2[math.AG]}.
\bibitem[CG72]{CG72}S. Clemens, P. Griffiths. The Intermediate Jacobian of the Cubic Threefold, Annals of Mathematics. Second Series 95 (2): 281356, (1972).
\bibitem[ChSh11]{ChSh11} I. Cheltsov, C. Shramov. Cremona groups and the Icosahedron. (preprint)
\bibitem[Ha77]{Ha77} R. Hartshorne. Algebraic Geometry. Springer 1977.
\bibitem[Isk78]{Isk78} V.A. Iskovskikh. Fano 3-folds I, II. Math USSR, Izv. \textbf{11}, 485-527 (1977); 12, 469-506 (1978).
\bibitem[Isk79]{Isk79} V.A. Iskovskikh. Birational Automorphisms of Three-Dimensional Algebraic Varieties.  Current problems in mathematics, VINITI, Moscow, 12, 159-236 (Russian). [English transl.: J. Soviet Math. 13 (1980) 815-868], Zbl. 428.14017.
\bibitem[IP99]{IP99} V.A. Iskovskikh, Yu.G. Prokhorov. Algebraic Geometry V: Fano varieties. Springer 1999.
\bibitem[JP06]{JP06} P. Jahnke, T. Peternell. Almost del Pezzo manifolds. \href{http://arxiv.org/PS_cache/math/pdf/0612/0612516v1.pdf}{arXiv:math/0612516v1 [math.AG]}.
\bibitem[JPR05]{JPR05} P Jahnke, T. Peternell, I. Radloff. Threefolds with Big and Nef AntiCanonical Bundles I. Math. Ann. 333, No.3, 569-631 (2005).
\bibitem[JPR07]{JPR07} P Jahnke, T. Peternell, I. Radloff. Threefolds with Big and Nef AntiCanonical Bundles II. 	 \href{http://arxiv.org/PS_cache/arxiv/pdf/0710/0710.2763v1.pdf}{arXiv:0710.2763v1 [math.AG]}.
\bibitem[Kal09]{Kal09} A-S Kaloghiros. A Classification of Terminal Quartic 3-Folds and Applications to Rationality Questions. \href{http://arxiv.org/PS_cache/arxiv/pdf/0908/0908.0289v1.pdf}{arXiv:0908.0289v1 [math.AG]}.
\bibitem[Kn02]{Kn02}A.L. Knutsen. Smooth Curves on Projective K3 Surfaces.  Math. Scand. 90 (2002), 215-231.
\bibitem[Ko89]{Ko89} J. Koll\'{a}r. Flops. Nagoya Math. J. 113, 15-36 (1989).
\bibitem[Ott92]{Ott92} O. Ottaviani. \textit{On Threefolds which are Scrolls"}. Ann. Scuola Norm. Sup. Pisa Cl. Sci. (4) 19 (1992), no. 3, 451471.
\bibitem[Puk88]{Puk88} A.V. Pukhlikov. \textit{Birational Automorphisms of a Three-Dimensional Quartic with a Simple Singularity}, Mat. Sb. 135, 472-495 (1988) (Russian). [English transl.: Math. USSR-SB. 63 (1989) 457-482, Zbl. 668.14007].
\bibitem[Shi89]{Shi89} K. Shin. 3-Dimensional Fano Varieties with Canonical Singularities.  \textit{Tokyo J. Math}., 12(2): 375-385, 1989.
\bibitem[StD94]{StD94} B. Saint-Donat. Projective Models of K3 surfaces, Ameri. J. Math. 96 (1974), 602-639.
\bibitem[Tak02]{Tak02} H. Takagi. On classification of $\Q$-Fano 3-folds of Gorenstein Index 2. I, II. \textit{Nagoya Math. J.}, 167:117-155,157-216, 2002.
\bibitem[Tak89]{Tak89} K. Takeuchi. Some birational maps of Fano 3-Folds \textit{Compositio Math}.,71(3): 265-283, 1989. 	
\bibitem[Tak09]{Tak09} K. Takeuchi. Weak Fano Threefolds with del Pezzo Fibration. \href{http://arxiv.org/PS_cache/arxiv/pdf/0910/0910.2188v1.pdf}{arXiv:0910.2188v1 [math.AG]}.




\end{thebibliography}
\end{document}